\documentclass[reqno,a4paper]{amsart}
	\usepackage[english,activeacute]{babel}
	\usepackage{amssymb,amsmath,amsthm,amsfonts,mathrsfs,hyperref,mathtools,stmaryrd}
	\usepackage[font=footnotesize]{caption}
	\usepackage{enumitem}
	
	\usepackage{tikz}
	\usetikzlibrary{shapes}
	\usetikzlibrary{shapes.geometric}
	\usetikzlibrary[topaths]
	\usetikzlibrary{arrows}
	\usepackage{tikz-network} %
	
	\definecolor{verydarkgreen}  {cmyk}{0.65,0.00,0.95,0.40}
	\definecolor{darkgreen}      {cmyk}{0.90,0.00,0.90,0.10}
	\definecolor{lightgreen}     {cmyk}{0.50,0.00,1.00,0.00}
	\definecolor{verylightgreen} {cmyk}{0.25,0.00,0.75,0.00}
	
	\usepackage{xcolor}
	\usepackage[numbers,square,sort]{natbib}
	\usepackage{tabularx}
	\usepackage{enumitem}  
	\usepackage[T1]{fontenc}
	\usepackage[foot]{amsaddr}
	\usepackage{url}
	\usepackage{verbatim}

	\theoremstyle{plain}
	\newtheorem{theorem}{Theorem}[section]
	
	\newtheorem{lemma}[theorem]{Lemma}
	
	\newtheorem{corollary}[theorem]{Corollary}
	\theoremstyle{definition}
	
	\newtheorem{definition}[theorem]{Definition}

	\theoremstyle{remark}
	\newtheorem{remark}{Remark}[section]
	
	\numberwithin{equation}{section}
	
	\newcommand{\Eb}  {{\mathbb E}}

	\newcommand{\Nb}  {{\mathbb N}}
	
	\newcommand{\Pb}  {{\mathbb P}}
	
	\newcommand{\Rb}  {{\mathbb R}}

	\newcommand{\Ns} {{\mathcal N}}
	
	\newcommand{\Ps} {{\mathcal P}}

	\newcommand{\Us} {{\mathcal U}}

	\newcommand{\cs}{{\mathfrak c}}
	
	\newcommand{\dd}{{\, \rm d}}
	\newcommand{\ee}{{\rm e}}
	\newcommand{\bee}{{\mathbf{e}}}
	
	\newcommand{\1}{{\bf 1}}
	\newcommand{\defeq}{\mathrel{\mathop:}=}
	\newcommand{\eqdef}{=\mathrel{\mathop:}}
	\newcommand{\EE}[1]{\mathbb{E} \left[#1\right]}
	\newcommand{\PP}[1]{\mathbb{P} \left[#1\right]}
	\newcommand{\Var}[1]{\textnormal{Var} \left[#1\right]}
	
	\newcommand{\tC}{\texttt{C}}
	\newcommand{\tla}{\boldsymbol{\lambda}}

	\newcommand{\bU}{\mathbf{U}}
	\newcommand{\bx}{\mathbf{x}}
	\newcommand{\bX}{\mathbf{X}}
	\newcommand{\bc}{\mathbf{c}}
	\newcommand{\bC}{\mathbf{C}}
	\newcommand{\bM}{\mathbf{M}}

	\newcommand{\bl}{\boldsymbol \ell}

	\date{\today}
	\title[]{On the block size spectrum of a class of exchangeable dynamic random graphs}
	\author{Frederic Alberti}\address{Frederic Alberti: fralbert@uni-mainz.de, Johannes Gutenberg Universität Mainz, Institut für Mathematik,  Germany} \author{Florin Boenkost} \address{Florin Boenkost: \,florin.boenkost@univie.ac.at, Universität Wien, Fakultät für Mathematik, Austria}  \author{Fernando Cordero} \address{Fernando Cordero: fernando.cordero@boku.ac.at, BOKU University, Institute of Mathematics, Department of Natural Sciences and Sustainable Ressources, Austria}

\begin{document}
		\begin{abstract}
			
			In this work we introduce the dynamic $\Theta$-random graph and the associated $\Theta$-coalescent with momentum. Dynamic $\Theta$-random graphs are a subclass of exchangeable and consistent random graph processes, parametrised by a measure $\Theta$ on $[0,1]\times (0,1]$, inspired by the classic $\Lambda$-coalescent from mathematical population genetics. The $\Theta$-coalescent with momentum accounts for the small connected components of this graph; in contrast to the underlying random graph it is exchangeable but not consistent. Our main results specialise on the case where $\Theta$ is the product of a beta measure and a Dirac mass at $1$. We prove a dynamic law of large numbers for the block size spectrum, which tracks the numbers of blocks containing $1,...,d$ elements. On top of that, we provide a functional limit theorem for the fluctuations. The limit process satisfies a stochastic differential equation of Ornstein-Uhlenbeck type.\\
			
			
		\end{abstract}
		
		\maketitle	
		
		\section{Introduction}
		
		The study of large random graphs and networks is an active and rapidly developing area of research~\cite{vdHofstadt2016,Durrett2007} and for example \cite{Berestycki2018,Borgs2019}. In addition to their rich mathematical structure, random networks find applications in a large number of fields such as biology~\cite{Koutrouli2020}, sociology~\cite{Feld1991}, neuroscience~\cite{Bressloff2014}, computer science~\cite{BrinPage1998}, and more recently and extensively in the context of neural networks \cite{Wu2021}.
		
		In mathematical population genetics, a particular kind of random tree structures, so-called 
		\emph{coalescent trees} play a crucial role in the description of ancestral structures. They are often expressed as Markov processes on the set of partitions of integers between $1$ and $n$, where $n$ represents the sample size. At any given time $t$, a partition
		$\Pi_t$ of $\{1,\ldots,n\}$ represents the ancestral relationships among genes in the sample; two samples belong to the same block, if and only if they have a common ancestor at (backward) time $t$. 
		
		Particular attention has been paid to coalescent processes that are \emph{exchangeable} in the sense that the underlying population is assumed to be sufficiently homogeneous, so that its evolution is invariant with regards to how the sample is labelled. Moreover, coalescent processes are often assumed to be \emph{consistent} in the sense that picking a smaller subsample of a larger sample is assumed to yield the same observation as starting with a smaller sample to begin with.
		A complete classification of coalescent processes that are consistent and exchangeable and do not exhibit multiple mergers has been carried out in~\cite{Sagitov1999,Pitman1999} 
		and is one of the cornerstones of coalescent theory~\cite{Wakeley2016}. Such a coalescent process is uniquely characterised by a finite measure $\Lambda$ on $[0,1]$. For $u > 0$, we encounter 
		a 
		\emph{$u$-merger}
		at rate $\Lambda(\dd u)/u^2$ in which
		each block is coloured independently with probability $u$. Subsequently, all coloured blocks merge into a single block. In addition, each pair of blocks merges independently at rate $\Lambda(\{ 0 \})$.
		Since then, many of their properties have been investigated, such as the phenomenon of coming down from infinity and asymptotic block counts, in particular in the special case where $\Lambda$ is the beta
		distribution~\cite{Limic2015,Miller2023,Schweinsberg2000CDI}.
		
		More generally, a similar classification of exchangeable and consistent graph-valued (Feller) processes has been carried out by
		H.~Crane~\cite{Crane2017}. He showed that such processes may only perform the following three types of transitions. First, they may perform \emph{edgewise} transitions at which edges (dis)appear independently at fixed rates. Secondly, at \emph{vertexwise} transitions, edges adjacent to a fixed vertex appear or disappear independently. Finally, they may perform \emph{global} transitions, which amount to the addition and removal of dense subgraphs, corresponding roughly speaking, to random graph limits or graphons~\cite{DiaconisJanson2008}.
		
		Global transitions, in particular, render exchangeable and consistent Feller processes on graphs an extremely rich class of stochastic processes, which makes its systematic exploration a challenging task.
		However, the special case where only edgewise transitions are performed and, moreover, edges only appear but do not disappear, has been well studied; this is the traditional setting of the dynamic Erd\H{o}s-Rényi model. It was shown that, after letting the total
		number of vertices tend to infinity, slowing down time and normalising appropriately, the evolution of the frequencies of small connected components is captured by a deterministic ordinary differential equation. This equation is known as \emph{Smoluchowski's coagulation equation}~\cite{Fournier2004}. Moreover, it was shown that the structure of connected components can be described by the excursions of a Brownian motion with drift; see~\cite{Broutin2015}, which generalises an earlier result by D.~Aldous~\cite{Aldous1997} on the static Erd\H{o}s-Rényi model.
		
		Inspired by the aforementioned $\Lambda$-coalescents from mathematical biology, we consider the natural subclass of exchangeable consistent dynamic random graphs that exhibits only edgewise and global transitions, but no vertexwise transitions. Moreover, like the dynamic Erd\H{o}s-Rényi graph, our model will be monotone in the sense that edges are only added, not deleted. Furthermore, we only allow for those global transitions that consist of the superposition of Bernoulli-random graphs. Put simply, these are transitions in which we colour each vertex independently with a certain probability $u$, and subsequently connect each pair of selected vertices with the same random probability $q$. This also has a natural interpretation in terms of the evolution of social networks; a global transition corresponds to a large meeting in which each individual participates with probability $u$ and each pair of participants has a certain probability $q$ of subsequently becoming friends.
		
		For our main results, we will furthermore take $\Theta$  of the form $\Lambda \otimes \delta_1^{}$, where $\Lambda$ is the Beta distribution,  which is in line with the tradition in population genetics and allows for nice and explicit calculations. We expect our results to also hold in (slightly) greater generality. Similar to~\cite{Miller2023}, we derive a dynamic law of large numbers for the number of connected components containing $1,...,d$ elements, see Remark \ref{rem:Theorem2.4}. We also prove a non-standard functional limit theorem for the fluctuations around this deterministic limit. Our results mirror those of~\cite{Limic2015} on the fluctuations of the total number of blocks in the $\Lambda$-coalescent when $\Lambda$ looks like a Beta distribution around the origin. More details can be found in Remark \ref{rem:limic}. Although our results are only stated for a subclass of $\Theta$-random graphs, the introduction of the more general class allows us to relate our results to the existing literature.
		
		The rest of the paper is organised as follows. 
		After introducing the model in full generality, we briefly discuss a possible application in the modelling of social bubbles. Then, we restrict to the setting of $\Theta$ being of the form
		$\Lambda \times \delta_1$, with $\Lambda$ being the beta measure. 
		Our main results are stated in Subsection~\ref{subsec:results} and we start Section~\ref{sect:proofs} by adapting the Poisson integral representation from~\cite{Limic2015} to our setting, which will be the central tool for our proof. These kinds of representation are reminiscent of those used in \cite{Kurtz1970} or \cite{EthierKurtz1986}. Right after, we give an overview over the structure of the proof before diving into the technical details. Finally, in Section~\ref{sect:calculus}, we recall some useful results from real analysis that are used elsewhere in the text.

		\subsection*{Notation} Before proceeding with the main text, let us fix some notation. In general we will use bold letters to denote vectors. Given $d \in \Nb$,
		we introduce three norms on $\Rb^d$, namely 
		\begin{equation*}
			|\bx| \defeq \sum_{i = 1}^d |x_i^{}| , \quad \|\bx  \|_2 \defeq \left( \sum_{i=1}^d x_i^2 \right)^{\frac{1}{2}}, \quad\textnormal{and}\quad
			\| \bx \| \defeq \sum_{i = 1}^d i |x_i^{} |.
		\end{equation*}
		We also write
		\begin{equation*}
			\Ps_{2+}(i) \defeq \big \{ \bl \in \Nb^d : \| \bl \| = i, \ell_i^{} = 0    \big \},
		\end{equation*}
		for the set of \emph{integer partitions} of $i$ into at least two parts. For any $\bc \in \frac{1}{n} \Rb^d$ and $\bl \in \Nb^d$, we make use of the abbreviation
		\begin{align*}
			\binom{n \bc}{\bl}:= \prod_{j=1}^d \binom{n c_j}{\ell_j}.
		\end{align*}
		Furthermore, we denote by $\bee_i^{}$ the 
		$i$-th canonical unit vector in 
		$\Rb^d$ so that 
		$\bc = \sum_{i=1}^d c_i^{} \bee_i^{}$.
		
		\section{Model and main results}
		
		\subsection{Two classes of exchangeable random graph processes}
		\begin{definition}\label{def:thetamuN}
			[$\Theta$-dynamic random graphs]
			Let $\Theta$ be a finite nonnegative measure on $[0,1]\times(0,1]$ and 
			\begin{align} \label{eq:smalltheta}
				\theta \defeq \int_{ \{ 0 \} \times [0,1]  } q   \, \Theta(\dd u, \dd q).
			\end{align}
			Let $N$ be a Poisson point process on $[0,\infty) \times (0,1]^2$ with intensity
			\begin{equation} \label{poissonintensity}
				\mu(\! \dd t, \dd u, \dd q) \defeq \dd t \frac{\Theta(\! \dd u, \dd q)}{u^2} \1_{ u >0}.
			\end{equation}
			The $\Theta$-dynamic random graph on $n$ vertices is the Markov process $G^n \coloneqq(G^n_t)_{t \geqslant 0}^{}$  on  the set of (undirected) graphs with vertices $[n] \defeq \{1,\ldots,n\}$, which evolves as follows.
			For each atom $(t,u,q)$ of $N$, $G^n$ performs at time $t$ a $(u,q)$-merger. This means that each vertex is coloured independently with probability $u$. Then, independently for each pair of coloured vertices that is not already joined by an edge, we draw an edge between them with probability $q$. In addition and independently of $(u,q)$-mergers, an edge appears independently at rate $\theta$ between each pair $i$ and $j$ that is not already joined by an edge.
		\end{definition}
		Note that the parametrisation in terms of a measure $\Theta$ is not unique. For instance, different $\Theta$ can yield the same pair-coalescence rate $\theta$ in
		Eq.~\eqref{eq:smalltheta}. An example of a transition of $G^n$ is given in Figure \ref{fig:1}. Although the Poisson point process $N$ can have an infinite number of points in sets of the form $[0,t)\times(0,1]^2$, a standard argument shows that, almost surely, only a finite number of them do actually result in a nontrivial transition in $G^n$. Therefore, the previous definition does indeed define a Markov process with càdlàg paths on the set of undirected graphs with $n$ vertices, if the latter is endowed with the discrete topology.
		
		The $\Theta$-dynamic random graph combines the classic dynamic Erd\H{o}s-Rényi (or Bernoulli) random graph with the possibility of large (dense) cliques emerging all at once. Note that this defines a consistent family $(G^n)_{n\in\Nb}$ of exchangeable dynamic random graphs, and thus, a Feller process taking values in the graphs with vertex set $\Nb$~\cite{Crane2017}. In general, $G_0^n$ might be any graph with $n$ vertices, however in our case we usually consider $G_0^n=([n],\varnothing)$.

		\begin{figure}[h]
			\scalebox{0.8}{
				\begin{minipage}[t]{0.25\textwidth}
					\begin{center}
						\begin{tikzpicture}[scale=0.45]
							\node[shade,shading=ball,circle,ball color=black!10!white, draw, minimum size=0.3cm, inner sep=0pt,opacity=1] (A) at ({360/16 * 0}:3cm) {\tiny{$5$}};
							\node[shade,shading=ball,circle,ball color=black!10!white,circle, draw, minimum size=0.3cm, inner sep=0pt,opacity=1] (B) at ({360/16 * 1}:3cm) {\tiny{$4$}};
							\node[shade,shading=ball,circle,ball color=green!60!white,circle, draw, minimum size=0.3cm, inner sep=0pt,opacity=1] (C) at ({360/16 * 2}:3cm) {\tiny{$3$}};
							\node[shade,shading=ball,circle,ball color=black!10!white,circle, draw, minimum size=0.3cm, inner sep=0pt,opacity=1] (D)at ({360/16 * 3}:3cm) {\tiny{$2$}};
							\node[shade,shading=ball,circle,ball color=black!10!white,circle, draw, minimum size=0.3cm, inner sep=0pt,opacity=1] (E) at ({360/16 * 4}:3cm) {\tiny{$1$}} ;
							\node[shade,shading=ball,circle,ball color=green!60!white,circle, draw, minimum size=0.3cm, inner sep=0pt] (F) at ({360/16 * 5}:3cm) {\tiny{$16$}};
							\node[shade,shading=ball,circle,ball color=black!10!white,circle, draw, minimum size=0.3cm, inner sep=0pt,opacity=1] (G) at ({360/16 * 6}:3cm) {\tiny{$15$}};
							\node[shade,shading=ball,circle,ball color=black!10!white,circle, draw, minimum size=0.3cm, inner sep=0pt,opacity=1] (H) at ({360/16 * 7}:3cm) {\tiny{$14$}};
							\node[shade,shading=ball,circle,ball color=black!10!white,circle, draw, minimum size=0.3cm, inner sep=0pt,opacity=1] (I) at ({360/16 * 8}:3cm) {\tiny{$13$}};
							\node[shade,shading=ball,circle,ball color=green!60!white,circle, draw, minimum size=0.3cm, inner sep=0pt] (J) at ({360/16 * 9}:3cm) {\tiny{$12$}};
							\node[shade,shading=ball,circle,ball color=green!60!white,circle, draw, minimum size=0.3cm, inner sep=0pt] (K) at ({360/16 * 10}:3cm) {\tiny{$11$}};
							\node[shade,shading=ball,circle,ball color=black!10!white,circle, draw, minimum size=0.3cm, inner sep=0pt,opacity=1] (L) at ({360/16 * 11}:3cm) {\tiny{$10$}};
							\node[shade,shading=ball,circle,ball color=black!10!white,circle, draw, minimum size=0.3cm, inner sep=0pt,opacity=1] (M) at ({360/16 * 12}:3cm){\tiny{$9$}};
							\node[shade,shading=ball,circle,ball color=black!10!white,circle, draw, minimum size=0.3cm, inner sep=0pt,opacity=1] (N) at ({360/16 * 13}:3cm){\tiny{$8$}};
							\node[shade,shading=ball,circle,ball color=black!10!white,circle, draw, minimum size=0.3cm, inner sep=0pt,opacity=1] (O)at ({360/16 * 14}:3cm) {\tiny{$7$}};
							\node[shade,shading=ball,circle,ball color=black!10!white,circle, draw, minimum size=0.3cm, inner sep=0pt,opacity=1] (P) at ({360/16 * 15}:3cm) {\tiny{$6$}};
							
							\Edge[bend=25,opacity=0.5](A)(C);
							\Edge[bend=-8,opacity=0.5](A)(P);
							\Edge[bend=-25,opacity=0.5](C)(P);
							
							\Edge[bend=8,opacity=0.5](F)(G);
							\Edge[bend=8,opacity=0.5](G)(H);
							\Edge[bend=8,opacity=0.5](H)(M);
							\Edge[bend=-8,opacity=0.5](M)(F);
							\Edge[bend=-8,opacity=0.5](M)(G);
							\Edge[bend=-25,opacity=0.5](H)(F);
							
							\Edge[bend=-8,opacity=0.5](D)(I); 
						\end{tikzpicture}
				\end{center}\end{minipage}
				\begin{minipage}[t]{0.1\textwidth}
					\begin{center}
						\begin{tikzpicture}
							\draw [line width=5pt, -stealth, white!80!black] (0,1) -- (1,1);
							\node[] (A) at (0,0) {};
						\end{tikzpicture}
					\end{center}
				\end{minipage}\begin{minipage}[t]{0.25\textwidth}
					\begin{center}
						\begin{tikzpicture}[scale=0.45]
							\node[shade,shading=ball,circle,ball color=black!10!white, draw, minimum size=0.3cm, inner sep=0pt,opacity=1] (A) at ({360/16 * 0}:3cm) {\tiny{$5$}};
							\node[shade,shading=ball,circle,ball color=green!60!white,circle, draw, minimum size=0.3cm, inner sep=0pt,opacity=1] (B) at ({360/16 * 1}:3cm) {\tiny{$4$}};
							\node[shade,shading=ball,circle,ball color=black!60!white,circle, draw, minimum size=0.3cm, inner sep=0pt,opacity=1] (C) at ({360/16 * 2}:3cm) {\tiny{$3$}};
							\node[shade,shading=ball,circle,ball color=black!10!white,circle, draw, minimum size=0.3cm, inner sep=0pt,opacity=1] (D)at ({360/16 * 3}:3cm) {\tiny{$2$}};
							\node[shade,shading=ball,circle,ball color=black!10!white,circle, draw, minimum size=0.3cm, inner sep=0pt,opacity=1] (E) at ({360/16 * 4}:3cm) {\tiny{$1$}} ;
							\node[shade,shading=ball,circle,ball color=black!60!white,circle, draw, minimum size=0.3cm, inner sep=0pt] (F) at ({360/16 * 5}:3cm) {\tiny{$16$}};
							\node[shade,shading=ball,circle,ball color=black!10!white,circle, draw, minimum size=0.3cm, inner sep=0pt,opacity=1] (G) at ({360/16 * 6}:3cm) {\tiny{$15$}};
							\node[shade,shading=ball,circle,ball color=black!10!white,circle, draw, minimum size=0.3cm, inner sep=0pt,opacity=1] (H) at ({360/16 * 7}:3cm) {\tiny{$14$}};
							\node[shade,shading=ball,circle,ball color=green!60!white,circle, draw, minimum size=0.3cm, inner sep=0pt,opacity=1] (I) at ({360/16 * 8}:3cm) {\tiny{$13$}};
							\node[shade,shading=ball,circle,ball color=black!60!white,circle, draw, minimum size=0.3cm, inner sep=0pt] (J) at ({360/16 * 9}:3cm) {\tiny{$12$}};
							\node[shade,shading=ball,circle,ball color=black!60!white,circle, draw, minimum size=0.3cm, inner sep=0pt] (K) at ({360/16 * 10}:3cm) {\tiny{$11$}};
							\node[shade,shading=ball,circle,ball color=black!10!white,circle, draw, minimum size=0.3cm, inner sep=0pt,opacity=1] (L) at ({360/16 * 11}:3cm) {\tiny{$10$}};
							\node[shade,shading=ball,circle,ball color=black!10!white,circle, draw, minimum size=0.3cm, inner sep=0pt,opacity=1] (M) at ({360/16 * 12}:3cm){\tiny{$9$}};
							\node[shade,shading=ball,circle,ball color=green!60!white,circle, draw, minimum size=0.3cm, inner sep=0pt,opacity=1] (N) at ({360/16 * 13}:3cm){\tiny{$8$}};
							\node[shade,shading=ball,circle,ball color=black!10!white,circle, draw, minimum size=0.3cm, inner sep=0pt,opacity=1] (O)at ({360/16 * 14}:3cm) {\tiny{$7$}};
							\node[shade,shading=ball,circle,ball color=black!10!white,circle, draw, minimum size=0.3cm, inner sep=0pt,opacity=1] (P) at ({360/16 * 15}:3cm) {\tiny{$6$}};
							
							\Edge[bend=25,opacity=0.5](A)(C);
							\Edge[bend=-8,opacity=0.5](A)(P);
							\Edge[bend=-25,opacity=0.5](C)(P);
							
							\Edge[bend=8,opacity=0.5](F)(G);
							\Edge[bend=8,opacity=0.5](G)(H);
							\Edge[bend=8,opacity=0.5](H)(M);
							\Edge[bend=-8,opacity=0.5](M)(F);
							\Edge[bend=-8,opacity=0.5](M)(G);
							\Edge[bend=-25,opacity=0.5](H)(F);
							
							\Edge[bend=-8,opacity=0.5](D)(I); 
							
							\Edge[bend=-8](K)(J);
							\Edge[bend=-8](K)(C);
							\Edge[bend=-8](K)(F); 
							\Edge[bend=-8](J)(F); 
							\Edge[bend=-8](F)(C);
							\Edge[bend=-8](J)(C);         
						\end{tikzpicture}
					\end{center}
					
				\end{minipage}\begin{minipage}[t]{0.1\textwidth}
					\begin{center}
						\begin{tikzpicture}
							\draw [line width=5pt, -stealth, white!80!black] (0,1) -- (1,1);
							\node[] (A) at (0,0) {};
						\end{tikzpicture}
					\end{center}
				\end{minipage}\begin{minipage}[t]{0.25\textwidth}
					\begin{center}
						\begin{tikzpicture}[scale=0.45]
							\node[shade,shading=ball,circle,ball color=black!10!white, draw, minimum size=0.3cm, inner sep=0pt,opacity=1] (A) at ({360/16 * 0}:3cm) {\tiny{$5$}};
							\node[shade,shading=ball,circle,ball color=black!60!white,circle, draw, minimum size=0.3cm, inner sep=0pt,opacity=1] (B) at ({360/16 * 1}:3cm) {\tiny{$4$}};
							\node[shade,shading=ball,circle,ball color=black!10!white,circle, draw, minimum size=0.3cm, inner sep=0pt,opacity=1] (C) at ({360/16 * 2}:3cm) {\tiny{$3$}};
							\node[shade,shading=ball,circle,ball color=black!10!white,circle, draw, minimum size=0.3cm, inner sep=0pt,opacity=1] (D)at ({360/16 * 3}:3cm) {\tiny{$2$}};
							\node[shade,shading=ball,circle,ball color=black!10!white,circle, draw, minimum size=0.3cm, inner sep=0pt,opacity=1] (E) at ({360/16 * 4}:3cm) {\tiny{$1$}} ;
							\node[shade,shading=ball,circle,ball color=black!10!white,circle, draw, minimum size=0.3cm, inner sep=0pt,opacity=1] (F) at ({360/16 * 5}:3cm) {\tiny{$16$}};
							\node[shade,shading=ball,circle,ball color=black!10!white,circle, draw, minimum size=0.3cm, inner sep=0pt,opacity=1] (G) at ({360/16 * 6}:3cm) {\tiny{$15$}};
							\node[shade,shading=ball,circle,ball color=black!10!white,circle, draw, minimum size=0.3cm, inner sep=0pt,opacity=1] (H) at ({360/16 * 7}:3cm) {\tiny{$14$}};
							\node[shade,shading=ball,circle,ball color=black!60!white,circle, draw, minimum size=0.3cm, inner sep=0pt,opacity=1] (I) at ({360/16 * 8}:3cm) {\tiny{$13$}};
							\node[shade,shading=ball,circle,ball color=black!10!white,circle, draw, minimum size=0.3cm, inner sep=0pt,opacity=1] (J) at ({360/16 * 9}:3cm) {\tiny{$12$}};
							\node[shade,shading=ball,circle,ball color=black!10!white,circle, draw, minimum size=0.3cm, inner sep=0pt,opacity=1] (K) at ({360/16 * 10}:3cm) {\tiny{$11$}};
							\node[shade,shading=ball,circle,ball color=black!10!white,circle, draw, minimum size=0.3cm, inner sep=0pt,opacity=1] (L) at ({360/16 * 11}:3cm) {\tiny{$10$}};
							\node[shade,shading=ball,circle,ball color=black!10!white,circle, draw, minimum size=0.3cm, inner sep=0pt,opacity=1] (M) at ({360/16 * 12}:3cm){\tiny{$9$}};
							\node[shade,shading=ball,circle,ball color=black!60!white,circle, draw, minimum size=0.3cm, inner sep=0pt,opacity=1] (N) at ({360/16 * 13}:3cm){\tiny{$8$}};
							\node[shade,shading=ball,circle,ball color=black!10!white,circle, draw, minimum size=0.3cm, inner sep=0pt,opacity=1] (O)at ({360/16 * 14}:3cm) {\tiny{$7$}};
							\node[shade,shading=ball,circle,ball color=black!10!white,circle, draw, minimum size=0.3cm, inner sep=0pt,opacity=1] (P) at ({360/16 * 15}:3cm) {\tiny{$6$}};
							
							\Edge[bend=25,opacity=0.5](A)(C);
							\Edge[bend=-8,opacity=0.5](A)(P);
							\Edge[bend=-25,opacity=0.5](C)(P);
							\Edge[bend=8,opacity=0.5](F)(G);
							\Edge[bend=8,opacity=0.5](G)(H);
							\Edge[bend=8,opacity=0.5](H)(M);
							\Edge[bend=-8,opacity=0.5](M)(F);
							\Edge[bend=-8,opacity=0.5](M)(G);
							\Edge[bend=-25,opacity=0.5](H)(F);
							\Edge[bend=-8,opacity=0.5](D)(I); 
							\Edge[bend=-8,opacity=0.5](K)(J);
							\Edge[bend=-8,opacity=0.5](K)(C);
							\Edge[bend=-8,opacity=0.5](K)(F); 
							\Edge[bend=-8,opacity=0.5](J)(F); 
							\Edge[bend=-8,opacity=0.5](F)(C);
							\Edge[bend=-8,opacity=0.5](J)(C);           
							\Edge[bend=-8,opacity=1](B)(I); 
							\Edge[bend=-8,opacity=1](B)(N);
							\Edge[bend=-8,opacity=1](N)(I);    
							
						\end{tikzpicture}
					\end{center}
					
				\end{minipage}
				
			}
			\vspace{.9cm}

			\scalebox{0.8}{ 
				\begin{minipage}[t]{0.25\textwidth}
					\begin{center}
						\begin{tikzpicture}[scale=0.45]
							\node[shade,shading=ball,circle,ball color=black!10!white, draw, minimum size=0.3cm, inner sep=0pt,opacity=1] (A) at ({360/16 * 0}:3cm) {\tiny{$5$}};
							\node[shade,shading=ball,circle,ball color=black!10!white,circle, draw, minimum size=0.3cm, inner sep=0pt,opacity=1] (B) at ({360/16 * 1}:3cm) {\tiny{$4$}};
							\node[shade,shading=ball,circle,ball color=green!60!white,circle, draw, minimum size=0.3cm, inner sep=0pt,opacity=1] (C) at ({360/16 * 2}:3cm) {\tiny{$3$}};
							\node[shade,shading=ball,circle,ball color=black!10!white,circle, draw, minimum size=0.3cm, inner sep=0pt,opacity=1] (D)at ({360/16 * 3}:3cm) {\tiny{$2$}};
							\node[shade,shading=ball,circle,ball color=black!10!white,circle, draw, minimum size=0.3cm, inner sep=0pt,opacity=1] (E) at ({360/16 * 4}:3cm) {\tiny{$1$}} ;
							\node[shade,shading=ball,circle,ball color=green!60!white,circle, draw, minimum size=0.3cm, inner sep=0pt] (F) at ({360/16 * 5}:3cm) {\tiny{$16$}};
							\node[shade,shading=ball,circle,ball color=black!10!white,circle, draw, minimum size=0.3cm, inner sep=0pt,opacity=1] (G) at ({360/16 * 6}:3cm) {\tiny{$15$}};
							\node[shade,shading=ball,circle,ball color=black!10!white,circle, draw, minimum size=0.3cm, inner sep=0pt,opacity=1] (H) at ({360/16 * 7}:3cm) {\tiny{$14$}};
							\node[shade,shading=ball,circle,ball color=black!10!white,circle, draw, minimum size=0.3cm, inner sep=0pt,opacity=1] (I) at ({360/16 * 8}:3cm) {\tiny{$13$}};
							\node[shade,shading=ball,circle,ball color=green!60!white,circle, draw, minimum size=0.3cm, inner sep=0pt] (J) at ({360/16 * 9}:3cm) {\tiny{$12$}};
							\node[shade,shading=ball,circle,ball color=green!60!white,circle, draw, minimum size=0.3cm, inner sep=0pt] (K) at ({360/16 * 10}:3cm) {\tiny{$11$}};
							\node[shade,shading=ball,circle,ball color=black!10!white,circle, draw, minimum size=0.3cm, inner sep=0pt,opacity=1] (L) at ({360/16 * 11}:3cm) {\tiny{$10$}};
							\node[shade,shading=ball,circle,ball color=black!10!white,circle, draw, minimum size=0.3cm, inner sep=0pt,opacity=1] (M) at ({360/16 * 12}:3cm){\tiny{$9$}};
							\node[shade,shading=ball,circle,ball color=black!10!white,circle, draw, minimum size=0.3cm, inner sep=0pt,opacity=1] (N) at ({360/16 * 13}:3cm){\tiny{$8$}};
							\node[shade,shading=ball,circle,ball color=black!10!white,circle, draw, minimum size=0.3cm, inner sep=0pt,opacity=1] (O)at ({360/16 * 14}:3cm) {\tiny{$7$}};
							\node[shade,shading=ball,circle,ball color=black!10!white,circle, draw, minimum size=0.3cm, inner sep=0pt,opacity=1] (P) at ({360/16 * 15}:3cm) {\tiny{$6$}};
							
							\Edge[bend=25,opacity=0.5](A)(C);
							\Edge[bend=-8,opacity=0.5](A)(P);
							
							\Edge[bend=8,opacity=0.5](G)(H);
							\Edge[bend=-8,opacity=0.5](M)(F);
							\Edge[bend=-25,opacity=0.5](H)(F);
							
							\Edge[bend=-8,opacity=0.5](D)(I); 
						\end{tikzpicture}
				\end{center}\end{minipage}\begin{minipage}[t]{0.1\textwidth}
					\begin{center}
						\begin{tikzpicture}
							\draw [line width=5pt, -stealth, white!80!black] (0,1) -- (1,1);
							\node[] (A) at (0,0) {};
						\end{tikzpicture}
					\end{center}
				\end{minipage}\begin{minipage}[t]{0.25\textwidth}
					\begin{center}
						\begin{tikzpicture}[scale=0.45]
							\node[shade,shading=ball,circle,ball color=black!10!white, draw, minimum size=0.3cm, inner sep=0pt,opacity=1] (A) at ({360/16 * 0}:3cm) {\tiny{$5$}};
							\node[shade,shading=ball,circle,ball color=green!60!white,circle, draw, minimum size=0.3cm, inner sep=0pt,opacity=1] (B) at ({360/16 * 1}:3cm) {\tiny{$4$}};
							\node[shade,shading=ball,circle,ball color=black!60!white,circle, draw, minimum size=0.3cm, inner sep=0pt,opacity=1] (C) at ({360/16 * 2}:3cm) {\tiny{$3$}};
							\node[shade,shading=ball,circle,ball color=black!10!white,circle, draw, minimum size=0.3cm, inner sep=0pt,opacity=1] (D)at ({360/16 * 3}:3cm) {\tiny{$2$}};
							\node[shade,shading=ball,circle,ball color=black!10!white,circle, draw, minimum size=0.3cm, inner sep=0pt,opacity=1] (E) at ({360/16 * 4}:3cm) {\tiny{$1$}} ;
							\node[shade,shading=ball,circle,ball color=black!60!white,circle, draw, minimum size=0.3cm, inner sep=0pt] (F) at ({360/16 * 5}:3cm) {\tiny{$16$}};
							\node[shade,shading=ball,circle,ball color=black!10!white,circle, draw, minimum size=0.3cm, inner sep=0pt,opacity=1] (G) at ({360/16 * 6}:3cm) {\tiny{$15$}};
							\node[shade,shading=ball,circle,ball color=black!10!white,circle, draw, minimum size=0.3cm, inner sep=0pt,opacity=1] (H) at ({360/16 * 7}:3cm) {\tiny{$14$}};
							\node[shade,shading=ball,circle,ball color=green!60!white,circle, draw, minimum size=0.3cm, inner sep=0pt,opacity=1] (I) at ({360/16 * 8}:3cm) {\tiny{$13$}};
							\node[shade,shading=ball,circle,ball color=black!60!white,circle, draw, minimum size=0.3cm, inner sep=0pt] (J) at ({360/16 * 9}:3cm) {\tiny{$12$}};
							\node[shade,shading=ball,circle,ball color=black!60!white,circle, draw, minimum size=0.3cm, inner sep=0pt] (K) at ({360/16 * 10}:3cm) {\tiny{$11$}};
							\node[shade,shading=ball,circle,ball color=black!10!white,circle, draw, minimum size=0.3cm, inner sep=0pt,opacity=1] (L) at ({360/16 * 11}:3cm) {\tiny{$10$}};
							\node[shade,shading=ball,circle,ball color=black!10!white,circle, draw, minimum size=0.3cm, inner sep=0pt,opacity=1] (M) at ({360/16 * 12}:3cm){\tiny{$9$}};
							\node[shade,shading=ball,circle,ball color=green!60!white,circle, draw, minimum size=0.3cm, inner sep=0pt,opacity=1] (N) at ({360/16 * 13}:3cm){\tiny{$8$}};
							\node[shade,shading=ball,circle,ball color=black!10!white,circle, draw, minimum size=0.3cm, inner sep=0pt,opacity=1] (O)at ({360/16 * 14}:3cm) {\tiny{$7$}};
							\node[shade,shading=ball,circle,ball color=black!10!white,circle, draw, minimum size=0.3cm, inner sep=0pt,opacity=1] (P) at ({360/16 * 15}:3cm) {\tiny{$6$}};
							
							\Edge[bend=25,opacity=0.5](A)(C);
							\Edge[bend=-8,opacity=0.5](A)(P);
							
							\Edge[bend=8,opacity=0.5](G)(H);
							\Edge[bend=-8,opacity=0.5](M)(F);
							\Edge[bend=-25,opacity=0.5](H)(F);
							
							\Edge[bend=-8,opacity=0.5](D)(I); 
							
							\Edge[bend=-8](K)(J);
							\Edge[bend=-8](K)(F); 
							\Edge[bend=-8](J)(C);         
						\end{tikzpicture}
					\end{center}
					
				\end{minipage}\begin{minipage}[t]{0.1\textwidth}
					\begin{center}
						\begin{tikzpicture}
							\draw [line width=5pt, -stealth, white!80!black] (0,1) -- (1,1);
							\node[] (A) at (0,0) {};
						\end{tikzpicture}
					\end{center}
				\end{minipage}\begin{minipage}[t]{0.25\textwidth}
					\begin{center}
						\begin{tikzpicture}[scale=0.45]
							\node[shade,shading=ball,circle,ball color=black!10!white, draw, minimum size=0.3cm, inner sep=0pt,opacity=1] (A) at ({360/16 * 0}:3cm) {\tiny{$5$}};
							\node[shade,shading=ball,circle,ball color=black!60!white,circle, draw, minimum size=0.3cm, inner sep=0pt,opacity=1] (B) at ({360/16 * 1}:3cm) {\tiny{$4$}};
							\node[shade,shading=ball,circle,ball color=black!10!white,circle, draw, minimum size=0.3cm, inner sep=0pt,opacity=1] (C) at ({360/16 * 2}:3cm) {\tiny{$3$}};
							\node[shade,shading=ball,circle,ball color=black!10!white,circle, draw, minimum size=0.3cm, inner sep=0pt,opacity=1] (D)at ({360/16 * 3}:3cm) {\tiny{$2$}};
							\node[shade,shading=ball,circle,ball color=black!10!white,circle, draw, minimum size=0.3cm, inner sep=0pt,opacity=1] (E) at ({360/16 * 4}:3cm) {\tiny{$1$}} ;
							\node[shade,shading=ball,circle,ball color=black!10!white,circle, draw, minimum size=0.3cm, inner sep=0pt,opacity=1] (F) at ({360/16 * 5}:3cm) {\tiny{$16$}};
							\node[shade,shading=ball,circle,ball color=black!10!white,circle, draw, minimum size=0.3cm, inner sep=0pt,opacity=1] (G) at ({360/16 * 6}:3cm) {\tiny{$15$}};
							\node[shade,shading=ball,circle,ball color=black!10!white,circle, draw, minimum size=0.3cm, inner sep=0pt,opacity=1] (H) at ({360/16 * 7}:3cm) {\tiny{$14$}};
							\node[shade,shading=ball,circle,ball color=black!60!white,circle, draw, minimum size=0.3cm, inner sep=0pt,opacity=1] (I) at ({360/16 * 8}:3cm) {\tiny{$13$}};
							\node[shade,shading=ball,circle,ball color=black!10!white,circle, draw, minimum size=0.3cm, inner sep=0pt,opacity=1] (J) at ({360/16 * 9}:3cm) {\tiny{$12$}};
							\node[shade,shading=ball,circle,ball color=black!10!white,circle, draw, minimum size=0.3cm, inner sep=0pt,opacity=1] (K) at ({360/16 * 10}:3cm) {\tiny{$11$}};
							\node[shade,shading=ball,circle,ball color=black!10!white,circle, draw, minimum size=0.3cm, inner sep=0pt,opacity=1] (L) at ({360/16 * 11}:3cm) {\tiny{$10$}};
							\node[shade,shading=ball,circle,ball color=black!10!white,circle, draw, minimum size=0.3cm, inner sep=0pt,opacity=1] (M) at ({360/16 * 12}:3cm){\tiny{$9$}};
							\node[shade,shading=ball,circle,ball color=black!60!white,circle, draw, minimum size=0.3cm, inner sep=0pt,opacity=1] (N) at ({360/16 * 13}:3cm){\tiny{$8$}};
							\node[shade,shading=ball,circle,ball color=black!10!white,circle, draw, minimum size=0.3cm, inner sep=0pt,opacity=1] (O)at ({360/16 * 14}:3cm) {\tiny{$7$}};
							\node[shade,shading=ball,circle,ball color=black!10!white,circle, draw, minimum size=0.3cm, inner sep=0pt,opacity=1] (P) at ({360/16 * 15}:3cm) {\tiny{$6$}};
							
							\Edge[bend=25,opacity=0.5](A)(C);
							\Edge[bend=-8,opacity=0.5](A)(P);
							\Edge[bend=8,opacity=0.5](G)(H);
							\Edge[bend=-8,opacity=0.5](M)(F);
							\Edge[bend=-25,opacity=0.5](H)(F);
							\Edge[bend=-8,opacity=0.5](D)(I); 
							\Edge[bend=-8,opacity=0.5](K)(J);
							\Edge[bend=-8,opacity=0.5](K)(F); 
							\Edge[bend=-8,opacity=0.5](J)(C);           
							\Edge[bend=-8,opacity=1](B)(I); 
							\Edge[bend=-8,opacity=1](B)(N);
							
						\end{tikzpicture}
					\end{center}
					
				\end{minipage}
			}
			\caption{\label{fig:1} 
				An illustration of two consecutive transitions of the $\Theta$-graph on $16$ vertices. Left: a given state of the graph. Independently of each other, vertices are coloured green with some probability $u$.
				Middle: each pair of coloured vertices is independently joined by an edge with some probability $q$. The same procedure is repeated to produce a second transition 
				(middle to right). The top panel shows the special case $q = 1$ (which will be our focus in what follows). In the bottom panel, $q < 1$.
				In the case $q=1$, the associated $\Theta$ coalescent with momentum (Def.~\ref{def:multlambda}) performs the transitions: \\  $ \text{ } \qquad \qquad \qquad \{ \{1\}, \{2,13\}, \{\textcolor{darkgreen}{3},5,6 \}, \{4 \}, \{7\}, \{8\}, \{9,14,15,\textcolor{darkgreen}{16}\},
				\{10\}, \{\textcolor{darkgreen}{11}\}, \{\textcolor{darkgreen}{12}\} $ \\
				$ \text{ } \qquad \qquad \qquad \qquad \to
				\{ \{1\}, \{2,\textcolor{darkgreen}{13}\}, \{3,5,6,9,11,12,14,15,16\}, \{\textcolor{darkgreen}{4}\}, \{7\}, \{\textcolor{darkgreen}{8}\}, 
				\{10\}    \}$
				\\  $ \text{ } \qquad \qquad \qquad \qquad \qquad
				\to \{   
				\{1\}, \{2,4,8,13\}, \{3,5,6,9,11,12,14,15,16\}, \{7\}, \{10\}
				\}.
				$
			}
		\end{figure}
		
		It is natural to associate with $G^n$ a coalescent process, i.e. a Markov process taking values in the partitions of $[n]$, corresponding to the connected components
		of $G^n$, which becomes coarser as connected components merge over time. Note that somewhat contrary to common usage, we think of connected components as collection of vertices, not as subgraphs.
		
		\begin{definition} \label{def:multlambda}
			Let $G^n$ be the $\Theta$-dynamic random graph. We call the process $\Pi^n \coloneqq (\Pi^n_t)_{t \geqslant 0}^{}$, with $\Pi^n_t$ being the set of sets of vertices associated with the connected components of $G^n_t$, the \emph{$\Theta$-coalescent with momentum} on $n$ vertices.
		\end{definition}
		In this setting, larger blocks coalesce faster which is why we call it a coalescent with momentum, see \eqref{multiplicativemergeprob}.
		Clearly, $\Pi^n$ is a Markov process in its own right; upon encountering a $(u,q)$-merger,  we colour
		each vertex in $[n]$ independently with probability $u$. Subsequently and independently between each pair of coloured vertices, we draw an edge with probability $q$. Finally, we merge all blocks that are connected by these edges. More precisely, for each maximal set $\{A_1,A_2,\ldots,A_k\}$ of subsets of
		$\Pi_{t-}^n$ such that for all $i$ and $j$ there exist $v_i^{} \in A_i$ and $v_j^{} \in A_j$ such that $v_i^{}$ and $v_j^{}$ are joined by an edge, we form $\Pi_t^n$ by replacing $A_1,\ldots,A_k$ with the single block
		$A_1 \cup \ldots \cup A_k$. Moreover, since our construction is exchangeable, the process recording only the sizes of connected components (see Def.~\ref{def:sizespectrum} below) up to a certain order is Markovian as well.

		\begin{definition} \label{def:sizespectrum}
			For any fixed $n,d \in\Nb_+$ and with $\Pi^{n}$ as in Definition~\ref{def:multlambda}, we call the process
			\newline
			\mbox{
				$\texttt{C}^{n} = (\texttt{C}^{n}_t)_{t \geqslant 0}^{} = (\texttt{C}^{n}_{t,1}, \ldots, \texttt{C}^{n}_{t,d} )_{t \geqslant 0}^{}$
			} with
			\begin{equation*}
				\texttt{C}^{n}_{t,i} \defeq \big | \{ A \in \Pi_{t}^{n} : |A| = i \}     \big |
			\end{equation*}
			the \emph{block size spectrum (up to order $d$)} of $\Pi^{n}$.
		\end{definition}
		
		In particular, if we consider the Erd\H{o}s-Rényi setting ($\Theta( (0,1]^2)=0$, $\theta >0$), the block size spectrum of $\Pi^n_t$ corresponds to a special case of the so-called Marcus-Lushnikov process \cite{Marcus1968,Lushnikov1973} where pairs of particles with masses $x$ and $y$ merge independently at  rate $K(x,y)/n$, with $K(x,y) = xy$.
		
		\subsection{A motivation: social bubbles}
		Consider a community consisting of $n$ individuals interacting with each other at social gatherings. Their timing and impact are represented by the constant $\theta \geqslant 0$ and the Poisson point process $N$ from Def.~\eqref{def:thetamuN} in the following way.
		If $(t,u,q) \in N$, a social gathering takes place at time $t$. Individuals attend this meeting independently of each other with probability $u$ and friendships form independently between each pair of participants with probability $q$. In addition, each pair of individuals meets independently of the others (and independently of large social gatherings) and becomes friends at rate $\theta$. We define the \emph{social bubble} of a group of individuals as the individuals themselves, their friends, the friends of their friends, and so on. For simplicity, we assume that no one knows anyone in the beginning ($t=0$). For a more precise description, we identify each individual in the community with a number $i\in[n]$ and consider the partition-valued process $\Pi^n$ started in
		$\Pi^n_0 = 
		\big \{ \{ i \} : i \in [n] \big \}
		$.
		The social bubble at time $t$ of a group $S_0\subset [n]$ of individuals is then given by
		$$S_t\coloneqq \bigcup\limits_{\substack{B\in\Pi_t^n\\B\cap S_0\neq \varnothing}}B,$$
		the set of all individuals that are friends of friends of $\ldots$ of individuals in $S_0$;
		see Fig. \ref{fig:Diagram}. Assume now that $S_0\subset [n]$ is a uniform sample of size $m_0$ and let $M_t^n$ denote the size of the social bubble at time $t$. In particular, $M_0^n=m_0$. The process $M^n\coloneqq (M_t^n)_{t\geqslant 0}^{}$ is clearly not Markovian. Assume for instance that $S_0 = \{1\}$ and that individuals $1$ and $4$ meet and become friends at some time $t > 0$. If no other event occurred in the interval $(0,t)$, we would have $S_t = \{1,4\}$, whereas if individuals $2$ and $4$ became friends during that time interval (and no other event occurred), we would have $S_t \supseteq \{1,2,4\}$, even though $S_{t-} = \{1\}$ in both cases.
		
		We therefore consider the auxiliary process $(\tla_t^{n},\tC_t^n)_{t\geqslant 0}^{}$, where $\tC_t^n=(\tC_{t,1}^n,\ldots,\tC_{t,n}^n)$ is the full block-size spectrum (see Definition~\eqref{def:sizespectrum}) of $\Pi_t^n$ and $\tla_t^{n}\coloneqq(\lambda_{t,1}^{n},\ldots,\lambda_{t,n}^{n})$ with
		$$\lambda_{t,i}^{n}\coloneqq|\{B\in\Pi_t^n: |B|=i,\, B\cap S_0\neq\varnothing\}|,$$
		
		\begin{figure}[b!]
			\scalebox{0.8}{
				\begin{minipage}[b]{0.5\textwidth}
					\centering
					\scalebox{0.5}{\begin{tikzpicture}
							
							\node[] at (-2,0) {$0$};
							\node[] at (-2,2) {$t_1$};
							\node[] at (-2,5) {$t_2$};
							\node[] at (-2,7) {$t_3$};
							\node[] at (-2,8) {$T$};
							
							\draw[opacity=1, dotted]  (-1,0) -- (-1,8);
							\draw[opacity=1, dotted]  (0,0) -- (0,8);
							\draw[opacity=1, dotted]  (1,0) -- (1,8);
							\draw[opacity=1, dotted]  (2,0) -- (2,8);
							\draw[opacity=1, dotted]  (3,0) -- (3,8);
							\draw[opacity=1, dotted]  (4,0) -- (4,8);
							\draw[opacity=1, dotted]  (5,0) -- (5,8);
							\draw[opacity=1, dotted]  (6,0) -- (6,8);
							\draw[opacity=1, dotted]  (7,0) -- (7,8);
							
							\node[below] at (-1,0) {$1$};
							\node[below] at (0,0) {$2$};
							\node[below] at (1,0) {$3$};
							\node[below] at (2,0) {$4$};
							\node[below] at (3,0) {$5$};
							\node[below] at (4,0) {$6$};
							\node[below] at (5,0) {$7$};
							\node[below] at (6,0) {$8$};
							\node[below] at (7,0) {$9$};
							
							\begin{scope}[every node/.style={circle,fill,thick,draw,scale=0.6}]
								\node (A1) at (0,2) {};
								\node (A2) at (1,2) {};
								\node (A3) at (3,2) {};
								
								\node (B1) at (1,5) {};
								\node (B2) at (2,5) {};
								\node (B3) at (5,5) {};
								\node (B4) at (7,5) {};
								
								\node (C1) at (2,7) {};
								\node (C2) at (4,7) {};
							\end{scope}
							
							\draw[bend right, very thick]  (A1) to node [auto] {} (A2);
							\draw[bend right, very thick] (A1) to node [auto] {} (A3);
							
							\draw[bend right, very thick] (B1) to node [auto] {} (B4);
							\draw[bend right, very thick]  (B2) to node [auto] {} (B3);
							
							\draw[bend right, very thick]  (C1) to node [auto] {} (C2);

					\end{tikzpicture}}
				\end{minipage}\begin{minipage}[b]{0.5\textwidth}
					\centering
					\scalebox{0.5}{\begin{tikzpicture}
							
							\node[] at (-2,0) {$0$};
							\node[] at (-2,2) {$t_1$};
							\node[] at (-2,5) {$t_2$};
							\node[] at (-2,7) {$t_3$};
							\node[] at (-2,8) {$T$};
							
							\draw[opacity=1, dotted]  (-1,0) -- (-1,8);
							\draw[opacity=1, dotted]  (0,0) -- (0,8);
							\draw[opacity=1, dotted]  (1,0) -- (1,8);
							\draw[opacity=1, dotted]  (2,0) -- (2,8);
							\draw[opacity=1, dotted]  (3,0) -- (3,8);
							\draw[opacity=1, dotted]  (4,0) -- (4,8);
							\draw[opacity=1, dotted]  (5,0) -- (5,8);
							\draw[opacity=1, dotted]  (6,0) -- (6,8);
							\draw[opacity=1, dotted]  (7,0) -- (7,8);
							
							\node[below] at (-1,0) {$1$};
							\node[below] at (0,0) {$2$};
							\node[below] at (1,0) {$3$};
							\node[below] at (2,0) {$4$};
							\node[below] at (3,0) {$5$};
							\node[below] at (4,0) {$6$};
							\node[below] at (5,0) {$7$};
							\node[below] at (6,0) {$8$};
							\node[below] at (7,0) {$9$};
							
							\begin{scope}[every node/.style={circle,fill,thick,draw, color=red,scale=0.6}]
								\node (A1) at (0,2) {};
								\node (A2) at (1,2) {};
								\node (A3) at (3,2) {};
								\node (B1) at (1,5) {};
								\node (B4) at (7,5) {};			
								
							\end{scope}
							\begin{scope}[every node/.style={circle,fill,thick,draw, color=blue,scale=0.6}]
								\node (B2) at (2,5) {};
								\node (B3) at (5,5) {};
								\node (C1) at (2,7) {};
								\node (C2) at (4,7) {};
							\end{scope}

							\begin{scope}[every node/.style={circle,fill,opacity=0.4,thick,draw, color=black,scale=0.6}]
								
								\node (S0) at (-1,0) {};
								\node (S1) at (0,0) {};
								\node[color=red,opacity=0.3] (S2) at (1,0) {};
								\node (S3) at (2,0) {};
								\node (S4) at (3,0) {};
								\node (S5) at (4,0) {};
								\node[color=blue,opacity=0.3] (S6) at (5,0) {};
								\node (S7) at (6,0) {};
								\node (S8) at (7,0) {};
								\node (T0) at (-1,8) {};
								\node[color=red,opacity=0.3] (T1) at (0,8) {};
								\node[color=red,opacity=0.3] (T2) at (1,8) {};
								\node[color=blue,opacity=0.3] (T3) at (2,8) {};
								\node[color=red,opacity=0.3] (T4) at (3,8) {};
								\node[color=blue,opacity=0.3] (T5) at (4,8) {};
								\node[color=blue,opacity=0.3] (T6) at (5,8) {};
								\node (T7) at (6,8) {};
								\node[color=red,opacity=0.3] (T8) at (7,8) {};

								\node[opacity=0.3] (Sc0) at (-1,7) {};
								\node[color=red,opacity=0.3] (Sc1) at (0,7) {};
								\node[color=red,opacity=0.3] (Sc2) at (1,7) {};
								\node[color=red,opacity=0.3] (Sc4) at (3,7) {};
								\node[color=blue,opacity=0.3] (Sc6) at (5,7) {};
								\node (Sc7) at (6,6) {};
								\node[color=red,opacity=0.3] (Sc8) at (7,7) {};
								
								\node (Sb0) at (-1,5) {};
								\node[color=red] (Sb1) at (0,5) {};
								\node (Sb5) at (4,5) {};
								\node[color=red] (Sb4) at (3,5) {};
								\node (Sb7) at (6,5) {};
								
								\node (Sa0) at (-1,2) {};
								\node (Sa3) at (2,2) {};
								\node (Sa5) at (4,2) {};
								\node[color=blue,opacity=0.3] (Sa6) at (5,2) {};
								\node (Sa7) at (6,2) {};
								\node (Sa8) at (7,2) {};
								
							\end{scope}
							
							\draw[bend right, color=red, very thick]  (A1) to node [auto] {} (A2);
							\draw[bend right, color=red, very thick] (A1) to node [auto] {} (A3);
							
							\draw[bend right, color=red,very thick] (B1) to node [auto] {} (B4);
							\draw[bend right,color=blue, very thick]  (B2) to node [auto] {} (B3);
							\draw[bend right,color=red, very thick, color=red, opacity=0.3] (Sb1) to node [auto] {} (B1);
							\draw[bend right, very thick, color=red, opacity=0.3] (Sb1) to node [auto] {} (Sb4);
							
							\draw[bend right, color=blue, very thick]  (C1) to node [auto] {} (C2);
							\draw[bend right, very thick, color=red, opacity=0.1] (Sc1) to node [auto] {} (Sc2);
							\draw[bend right, dotted, color=red, opacity=0.1] (Sc2) to node [auto] {} (Sc4);
							\draw[bend right, very thick, color=red, opacity=0.1] (Sc1) to node [auto] {} (Sc4);
							\draw[bend right, color=red, very thick, opacity=0.3]  (Sc2) to node [auto] {} (Sc8); 
							\draw[bend right, color=blue, very thick, opacity=0.3]  (C1) to node [auto] {} (Sc6);
							
							\draw[opacity=1, color=blue, dotted]  (S6) to node [auto] {} (Sa6);
							\draw[opacity=1, color=blue, dotted]  (Sa6) to node [below] {} (B3);
							\draw[opacity=1, color=blue, dotted] (B3) to node [below] {} (Sc6) -- (T6); 	  
							\draw[opacity=1, color=blue, dotted]  (B2) -- (C1) -- (T3);
							\draw[opacity=1, color=red, dotted]  (S2) -- (A2) -- (B1) -- (Sc2) -- (T2);
							\draw[opacity=1, color=red, dotted]  (A1) -- (Sb1) -- (Sc1) -- (T1);
							\draw[opacity=1, color=red, dotted]  (A3) -- (Sb4) -- (Sc4) -- (T4);	
							\draw[opacity=1, color=red, dotted]  (B4) -- (Sc8) -- (T8);	          
							\draw[opacity=1, color=blue, dotted]  (C2) -- (T5);
							
					\end{tikzpicture}}
					
				\end{minipage}
			}
			\caption{Left: a graphical representation of the occurrence of social gatherings and the friendship relations established on their occasion in a community consisting of $9$ individuals; social gatherings occur at times $t_1, t_2, t_3$; participants are depicted as vertices of a graph; pairs of participants are connected by an edge if they become friends during the meeting. Right: the social network process associated to individual $3$ (resp. $7$) corresponding to the social gatherings in the left panel; the social gatherings are superposed on the diagram: at the times of a meeting, participants are depicted in dark red or dark blue; light red (resp. light blue) vertices represent individuals that belong to the social network of individual $3$ (resp. $7$), but that are not part of the meeting; light red and light blue edges represent connections established in previous meetings. At any time, vertices that are coloured red (resp. blue) belong to the social network of individual $3$ (resp. $7$).}
			\label{fig:Diagram}
		\end{figure}
		i.e. $\lambda_{t,i}^n$ is the number of blocks with size $i$ contained in the social bubble of $S_0$ at time $t$.

		Note that 
		$M_t = |S_t| = \sum_{i=1}^n i \lambda_{t,i}^n$.
		Owing to the exchangeability of the underlying graph process,
		$(\tla_t^n, \tC_t^n)_{t \geqslant 0}^{}$
		is actually a Markov process; assume that we observe a merger of $k$ (say, at time $t$) blocks $A_1,\ldots,A_k$ into one, $\ell_i^{}$ of which 
		contain exactly $i$ vertices so that
		$k = \ell_1^{} + \ldots + \ell_n^{}$. By exchangeability, the participating blocks are uniformly (with replacement) distributed. Since we also assumed that 
		$S_0$ is a uniform sample of size $m_0^{}$, this means that the numbers $V_i$ of blocks of size $i$ that participate in the merger \emph{and} contain an individual inside $S_0$ are independent for different $i$ and hypergeometrically distributed. More precisely, 
		\begin{equation*}
			\PP{V_i = v_i^{}\mid \tC_{t-,i}^n,\lambda_{t-,i}^n }
			=
			\frac{
				\binom{\tC_{t-,i}^n -\lambda_{t-,i}^n }{\ell_i^{} - v_i^{}}
				\binom{\lambda_{t-,i}^n}{v_i^{}}
			}
			{
				\binom{\tC_{t-,i}^n}{\ell_i^{}}
			}.
		\end{equation*}
		Clearly, the values of the $V_i^{}$ 
		determine the transition of 
		$(\lambda_t^n,\tC_t^n)$; we have
		\begin{equation*}
			\tC_{t}^n 
			= \tC_{t-}^n 
			+
			\bee_{\|\bl\|}^{}
			-
			\ell_1^{} \bee_1^{} - \ldots - \ell_n^{} \bee_n^{}.
		\end{equation*}
		and 
		\begin{equation*}
			\tla_t^n = \tla_{t-}^n 
			+
			\1_{|{\bf V}| > 0 } \,
			(\bee_{\| \bl \|}^{}
			- 
			V^{}_1 \bee_1^{} - \ldots - V^{}_n \bee_n^{}),
		\end{equation*}
		where $\bl = (\ell_1^{},\ldots,\ell_n^{})$
		and
		${ \bf V} = (V_1,\ldots,V_n)$.
		
		In fact, we can determine the distribution of $\tla^n_t$ conditional on the full block size spectrum $\tC_t^n$. More precisely, one can show that (see Appendix \ref{app-bubble} for the details)
		\begin{equation}
			\label{samplingformula1}
			\PP{\tla_t^{n}=\bl\mid \tC_t^n=\bc}=\frac{ \prod\limits_{i\in[n]}\binom{c_i}{\ell_i}\,  i^{\ell_i} }{\binom{n}{m_0}}  \Eb\bigg[\binom{\lVert \bl \rVert-F_{\bl}}{m_0-|\bl|}\bigg]1_{\{|\bl|\leqslant m_0
				\leqslant \lVert\bl\rVert,\, \bl \leqslant \bc \}}, 
		\end{equation}
		where $F_{\bl}=\sum_{i=1}^n\sum_{j=1}^{\ell_i}F_{i,j}$, and
		the random variables $(F_{i,j}:i\in[n],j\in[\ell_i])$ are independent and $F_{i,j}$
		is uniformly distributed on $[i]$, where $\bl \leqslant \bc$ is to be understood componentwise. 
		
		In particular, we have
		$$\PP{M_t^{n}=m}=\sum_{\lVert\bl\rVert=m}\frac{ \Eb\left[ \prod\limits_{i\in[n]}\binom{\tC_{t,i}^n}{\ell_i}\, i^{\ell_i}  \binom{\lVert\bl\rVert-F_{\bl}}{m_0-|\bl|}\right]}{\binom{n}{m_0}}, \quad m_0\leqslant m\leqslant n,$$
		where the $F_{\bl}$ is assumed to be independent of $\tC_t^n$ under $\Pb$.
		For example, with $m_0=1$, this yields $\PP{M_t^{n}=m}=\Eb[\tC_{t,m}^n]/n$.
		\subsection{$\Lambda$-dynamic random graphs and $\Lambda$-coalescents with momentum}
		
		For a finite nonnegative measure $\Lambda$ on $[0,1]$, the 
		\emph{$\Lambda$-dynamic random graph on $n$ vertices}
		is the $\Theta$-dynamic random graph on $n$ vertices with $\Theta(\dd u, \dd q) = \Lambda(\dd u) \delta_1(\dd q)$. By Eq.~\eqref{eq:smalltheta}, we then have
		$\theta = \Lambda(\{0\})$,
		see the top of Figure \ref{fig:1} for an illustration. 
		We call the associated coalescent process 
		$\Pi^n$ (see Definition~\ref{def:multlambda})
		the
		\emph{$\Lambda$-coalescent with momentum}.
		In the case of $\Lambda$-graphs, by a slight abuse of notation, we drop the last component in $N$ and $\mu$, i.e. we say that $N$ is a Poisson point process on $[0,\infty) \times (0,1]$ with intensity $\mu$, where
		\begin{align}
			\mu( \dd t , \dd u )= \dd t  \frac{\Lambda(\dd u)}{ u^2} \1_{(0,1]}(u). \label{eq:definition_of_N}
		\end{align}

		In this special case, where $q$ is always equal to $1$, we speak of $u$-mergers instead of $(u,1)$-mergers. Upon a $u$-merger, simply colour each vertex with probability $u$ and merge all blocks that contain at least one coloured vertex. 
		Equivalently,  we may skip the colouring step and mark each block $A$ independently with probability 
		\begin{equation} \label{markingprobability}
			p_{|A|}^{} (u) \defeq 1 - (1 - u)^{|A|} = |A| u + O(u^2).
		\end{equation} 
		With this, the probability that a given collection of $m$ blocks
		$A_1, \ldots, A_m$ participates in a $u$-merger is 
		\begin{equation} \label{multiplicativemergeprob}
			\prod_{i=1}^m p_{|A_i|}^{} (u) = \prod_{i = 1}^m \big ( 1 - (1-u)^{|A_i|}   \big ) = u^m \prod_{i = 1}^m |A_i| + O(u^{m+1}).
		\end{equation}
		Moreover, each pair $A,B$ of blocks in the $\Lambda$-coalescent with momentum merges independently at rate $\theta |A||B|$. This is in contrast to the classic $\Lambda$-coalescents studied in mathematical population genetics~\cite{Pitman1999, Sagitov1999}; there, upon a $u$-merger, each block is marked independently  with probability $u$, regardless of its size. In particular, the probability that a given collection of $m$ blocks $A_1,\ldots, A_m$ participates in a $u$-merger is $u^m$.
		For the general $\Theta$-coalescent, the analogue of Eq.~\eqref{multiplicativemergeprob} is significantly more complicated and we refrain from stating a precise formula here.

		To the best of our knowledge, the $\Lambda$-dynamic random graph model and the
		$\Lambda$-coalescent with momentum as defined here are new and have not been considered in this formulation. Our model is inspired by the construction of the classic $\Lambda$-coalescent known in mathematical biology and physics \cite{Berestycki2009}. The crucial difference lies in the fact that the rate at which two blocks get connected in a $\Lambda$-coalescent does not depend on the size of blocks, whereas in our case  bigger blocks are connected at a faster rate.
		
		\begin{remark}[Comparison to other random graph models]
			A static random graph, the so-called superposition of Bernoulli random graphs, has been studied for example in \cite{Bloznelis2024}; see also the references therein. In each step a random set $X$ of vertices is sampled and then a Bernoulli random graph is generated on $X$ with a random edge probability $q$. This procedure is repeated $m$ times and the final random graph is obtained by a superposition of these $m$ random graphs. From a certain point of view, this procedure mimics a static construction of $G^n_t$ if $t$ is chosen such that $G^n_t$ has been updated $m$ times. However, in our case the subsample of vertices considered in each step is a random variable which depends (implicitly) on $n$, whereas in \cite{Bloznelis2024} the distribution of $X$ (describing the random sampled subset) is fixed in $n$. Nevertheless, \cite{Bloznelis2024} consider similar questions as we do here and we discuss their results and compare them to our results in more detail in Remark \ref{rem:Bloznelis}.
			
			Another random (hyper)-graph model similar to ours is the one studied in \cite{Lemaire2017}. In this work, the authors analyse a dynamically evolving random graph, where at each time of a Poisson process, a finite random subset of vertices is selected uniformly, and a complete graph is drawn between them. Then, the random graph is constructed as a superposition of this complete graph and the prior state of the random graph.
			This can be seen as a dynamic version of the superposition of Bernoulli-random graphs in~\cite{Bloznelis2024} in the special case $q = 1$.
			Again this resembles our model, where the key difference lies in the fact that the sizes of the connected components added in \cite{Lemaire2017} does not depend on $n$, whereas in our model the coalescence rates grow with the sizes of the blocks involved.
			
			In the literature surrounding random graphs the block size spectrum is not considered often, but instead the so-called subgraph counts. That is counting the number of occasions a given graph $F$ appears in $G^n$, see for example \cite{Janson1990,Rucinski1988}. The block size frequency is a similar random graph statistics, however it counts the number of connected graphs appearing in $G^n$, which have $1,...,d$ vertices. Therefore, the block size frequency is a coarser statistics than the random graph counts and stops being informative once the graph becomes connected.
		\end{remark}

		\subsection{Results for the beta-dynamic random graph}
		\label{subsec:betacase} \label{subsec:results}
		In the following, we consider a special case of $\Lambda$-dynamic graphs, where $\Lambda$ is proportional to the Beta distribution, i.e.
		\begin{equation} \label{betadist}
			\Lambda (\dd u) \defeq \1_{[0,1]}(u) u^{\alpha - 1} (1 - u)^{\beta - 1} \dd u,
		\end{equation}
		for some fixed $\alpha \in (0,1)$ and $\beta > 0$, neglecting the normalisation $B(\alpha,\beta)^{-1}$ for ease of notation. This implies in particular, that we connect all chosen vertices with a complete graph and that we have $\theta =0$.

		
		We want to study the asymptotics of the block size spectrum as $n$ tends to infinity. It is important to keep in mind that in the parameter regime we are considering here (i.e., $\alpha \in (0,1)$), the classic Beta coalescent comes down from infinity, meaning that, even when started with infinitely many blocks at time $0$, it consists of a finite number of blocks of infinite size at every positive time. This implies that the rate of merging explodes and hence one has to slow down time accordingly to observe a non-trivial limit, see e.g. \cite{Miller2023} or \cite{Berestycki2010}. Since mergers occur with higher probability in the multiplicative setting, as seen in Eq.~\eqref{multiplicativemergeprob}, this will also be true for us.
		
		Let us be a bit more precise. For a single vertex to participate in a non-silent $u$-merger, it must be coloured (which happens with probability $u$) and at least one other vertex needs to be coloured as well 
		(which happens with probability $1 - (1-u)^n$). Thus, the rate at which each individual vertex is affected by a non-silent $u$-merger is
		\begin{align*}
			\int_{(0,1]} u (1-(1-u)^{n-1}) \frac{\Lambda(\dd u)}{u^2} &= \int_{(0,1]} u^{\alpha-2} (1-(1-u)^{n-1}) (1-u)^{\beta-1} \dd u \\
			&= n^{1-\alpha} \frac{\Gamma(\alpha)}{1-\alpha} + O(1),
		\end{align*}
		due to Lemma \ref{lem:integralestimate2}.
		Since there are $n$ vertices in total, this means that after slowing down time by a factor of the order 
		$n^{\alpha - 1}$, we expect to see a macroscopic (i.e. of order $n$) number of vertices involved in merging events per unit time. This motivates the following definition.
		
		\begin{definition}
			In what follows, we denote by 
			\begin{equation*}
				\bC^{n}_t \defeq \frac{1}{n} 
				\texttt{C}^{n}_{n^{\alpha - 1} t} 
			\end{equation*}
			the \emph{rescaled and normalised block size spectrum} of $\Pi^{n}$. 
		\end{definition}
		
		Before proceeding with our further investigation of $\bC^n$ and stating our main results, namely a law of large numbers and a functional limit theorem, note that $\bC^n$ is a Markov chain in its own right with state space
		\begin{equation*}
			E^n \defeq \Big \{ \bc \in [0,1]^d \cap \frac{1}{n} \Nb^d : \sum_{j = 1}^d j c_j^{} \leqslant 1   \Big \}.
		\end{equation*}
		Note that the family $(\Pi^{n})_{n \in \Nb_+}$ is not consistent, whence we cannot define the multiplicative $\Lambda$-coalescent on $\Nb_+$. 
		On the other hand, the family $(G^{n})_{n \in \Nb_+}$ of underlying graph processes is exchangeable as well as consistent, which will allow us to couple $\Pi^{n}$ for different $n$ in a natural way. Such a coupling will play an important role in our proofs; see Subsection~\ref{subsec:poissonrep} for a precise definition.

		
		Our first goal is to derive a dynamic law of large numbers for the normalised block size spectrum. That is, we will describe the limit as $n \to \infty$ of 
		$\bC^{n}$ in terms of an ordinary differential equation. 
		\begin{theorem}[A dynamic law of large numbers] \label{thm:lln}
			For all $\varepsilon > 0$ and $T > 0$, 
			\begin{equation*}
				\lim_{n \to \infty} \PP{\sup_{t \in [0,T]} |\bC^n_t - \cs_{t} |  > \varepsilon} \to 0, \quad  \textnormal{as } n \to \infty,
			\end{equation*}
			where $\cs_t = (\cs_{t,1}^{}, \ldots, \cs_{t,d}^{})_{t \geqslant 0}$ solves the following system of ordinary differential equations
			\begin{align} \label{eq:system_ODE}
				\frac{\dd}{\dd t} \cs_{t,i}^{} = \sum_{\bl \in \Ps_{2+}(i)} \prod_{j = 1}^d \frac{(j \cs_{t,j}^{})^{\ell_j^{}}}{\ell_j^{} !} \Gamma (\alpha + |\bl| - 2)
				- \frac{\Gamma(\alpha)}{1 - \alpha} i \cs_{t,i}^{} \eqdef F_i^{} (\cs_t^{}), \quad 1 \leqslant i \leqslant d,
			\end{align}
			with initial condition $\cs_{0,i} = \delta_{i,1}$.
		\end{theorem}
		
		\begin{remark} \label{rem:Theorem2.4} 
			Letting $\cs_0\in [0,1]^d$ such that $\| \cs_0 \| \leqslant 1$, then one can easily relax the assumption on the initial condition and simply assume $\EE{(\bC_0^n-\cs_0)^2}\to 0$ for $n \to \infty$. The initial condition of the system of ordinary differential equations in \eqref{eq:system_ODE} then has initial condition $\cs_0$.
			
			A similar result on the block size spectrum of the Beta coalescent coming down from infinity has recently been obtained in \cite{Miller2023}. They obtain the convergence of the block size spectrum towards polynomials, hence the decay in the frequency of blocks of any size is of polynomial order, whereas in our case the decay of blocks is exponentially fast. Notably, the time rescaling in both models is the same, namely time is slowed down by $n^{a-1}$. Since the Beta coalescent comes down from infinity for $a\in(0,1)$ \cite{Schweinsberg2000CDI}, Miller and Pitters also provide a limiting result if one starts the coalescent with infinitely many lineages. In our case the multiplicative coalescent is not consistent in $n$, hence we do not provide such a result.
		\end{remark}
		It is a straightforward exercise to compute $\cs$ in an iterative fashion.
		\begin{corollary}
			For $i = 1,\ldots,d$ and $t \geqslant 0$,
			\begin{equation*}
				\cs_{t,i}^{} = p_i^{} (t) \ee^{-i \gamma t},
			\end{equation*}
			where $\gamma = \Gamma(\alpha) / (1 - \alpha)$ and $p_i^{}$ is a polynomial of degree $i-1$ for each $i$ between $1$ and $d$, see Figure \ref{fig:2} for a plot of $\cs_{t,1}^{},\dots, \cs_{t,4}^{}$. These polynomials can be computed via the recursion 
			\begin{equation*}
				p_i^{} (t) = \cs_{0,i}^{} + \sum_{\bl \in \Ps_{2+} (i) } \Gamma ( \alpha + |\bl| - 2) \prod_{j = 1}^d \frac{j^{\ell_j}}{\ell_j !}
				\int_0^t \prod_{j = 1}^d p_j(s)^{\ell_j} \dd s.
			\end{equation*}
			In particular $p_1(t)=\cs_{0,1}$ and
			\begin{align*}
				p_2(t)&= \cs_{0,2} + \frac{\Gamma(\alpha)}{2} \cs_{0,1}^2 t\\
				p_3(t)&= \cs_{0,3} + \Gamma(\alpha+1)\left( \frac{\cs_{0,1}^3 t}{3!} + 2 \cs_{0,1} \cs_{0,2} t + \Gamma(\alpha) \cs_{0,1}^3 \frac{t^2}{4} \right).
			\end{align*}
		\end{corollary}
		\begin{figure}
			\centering
			\includegraphics[scale=0.65]{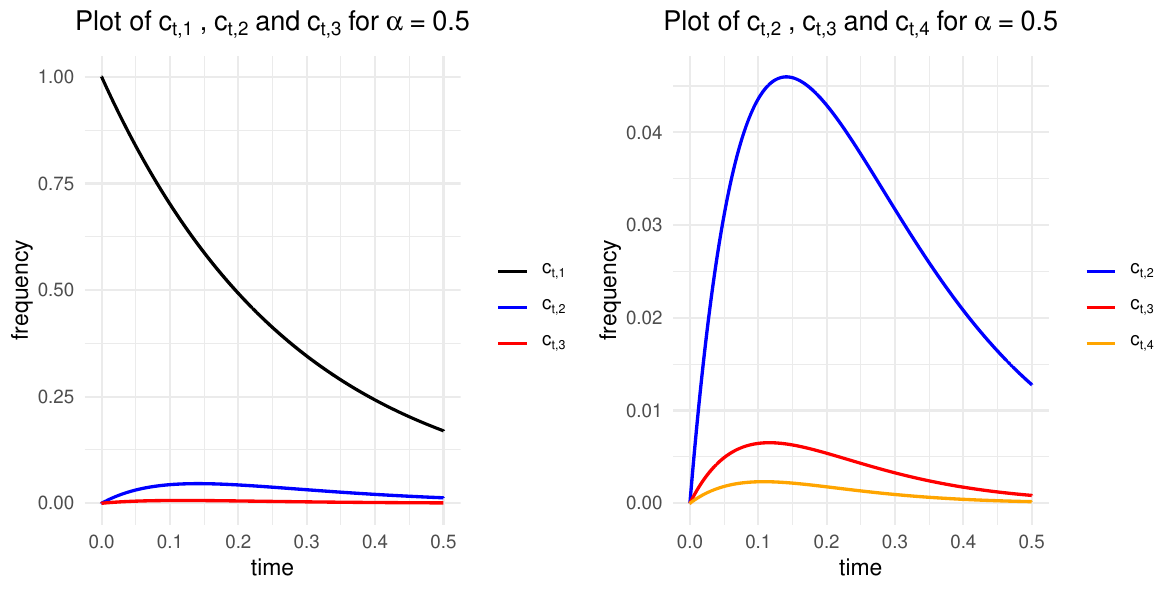}
			\caption{\label{fig:2} A plot of $\cs_{t,i}$ for $i=1,...,4$ for $\alpha=0.5$. The right hand side is a scaled version of the first picture excluding the blocks which contain exactly one element.}
		\end{figure}
		
		\begin{remark}
			Theorem~\ref{thm:lln} can be proved by computing the transition rates and using general theory~\cite{Kurtz1970,EthierKurtz1986}. However, in view of our proof of the functional limit theorem, we will work with a representation via Poisson integrals, see Subsection~\ref{subsec:poissonrep}, which also provides the appropriate coupling for $(\bC^n)_{n \geqslant 0}$ so that the convergence in probability in 
			Theorem~\ref{thm:fclt} holds.
			In fact, using this method, we obtain slightly more; we can show that for all $\varepsilon > 0$ and any $T > 0$, 
			\begin{equation*}
				\PP{\sup_{0 \leqslant t \leqslant T} |\bC^n_t - \cs_t^{}| \geqslant \varepsilon} = O(n^{\alpha - 1}) \, \textnormal{ as } n \to \infty.
			\end{equation*}
		\end{remark}
		The following theorem is a $L_2$ version of the law of large numbers, which plays an important role in our proof of the functional limit theorem below.
		\begin{theorem}[Law of large numbers --- $L_2$-version] \label{thm:llnl2}
			For all $T > 0$,
			\begin{equation*}
				\sup_{0 \leqslant t \leqslant T} \EE{ \| \bC_t^n - \cs_t^{}  \big \|_2^2} = O(n^{\alpha -1}).
			\end{equation*}
			
		\end{theorem}
		Note that Theorem \ref{thm:llnl2} does not directly imply Theorem \ref{thm:lln}.

		Our second main result is a functional limit theorem (FLT) for the fluctuations of $\bC^n$ around its deterministic limit $\cs$.
		
		\begin{theorem}[Functional Limit Theorem] \label{thm:fclt}
			For $n \in \Nb_+$, set $\sigma_n^{} \defeq n^{\frac{1-\alpha}{2-\alpha}}$ and $\bU_t^n \defeq \sigma_n^{} (\bC_t^n - \cs_t)$.
			Then, there exists a Poisson point process $\Ns$ on $[0,\infty) \times (0, \infty)$ with intensity
			\begin{equation*}
				\dd t \, u^{\alpha - 3}  \dd u,
			\end{equation*}
			defined on the same probability space as $\bC^n$,
			such that for any $T>0$
			\begin{equation*}
				\PP{ \sup_{0 \leqslant t \leqslant T} |U^n_{t,i} - U_{t,i}^{} | \geqslant \varepsilon} \to 0 \, \textnormal{as } n \to \infty.
			\end{equation*}
			Here, for all $i = 1,\ldots,d$, $\bU_i = (U_{t,i})_{t \geqslant 0}$ is a generalised Ornstein-Uhlenbeck process satisfying 
			\begin{equation*}
				\dd U_{t,i} = \langle \nabla F_i^{} (\cs_t^{}), \bU_t \rangle \dd t  + \dd M_{t,i}    ,
			\end{equation*}
			see \eqref{eq:system_ODE} for the definition of $F_i^{}$; the driving martingale $\bM_{i}$ is defined as
			\begin{equation*}
				M_{t,i} \defeq -i \int_{[0,t) \times [0, \infty)} \cs_{s,i} u \, \widetilde \Ns (\dd s \dd u),
			\end{equation*}
			where here and in the following $\widetilde{\Ns}$ denotes the compensation of $\Ns$.
			Put differently, $U_i$ satisfies
			\begin{equation*}
				\dd U_{t,i} = \langle \nabla F_i^{} (\cs_t^{}), \bU_t^{} \rangle \dd t - i \cs_{t,i}^{} \dd L_t ,
			\end{equation*}
			where $L$ is a spectrally positive Lévy process with Lévy measure $u^{\alpha - 3} \dd u$ driven by $\Ns$.
		\end{theorem}
		
		\begin{remark} \label{rem:limic}
			We expect our results to carry over to the more general situation that $\Lambda(\{0\}) = \Lambda(\{1\}) = 0$ and  there exists
			$y_0^{} \leqslant 1$ such that 
			\begin{align}
				\Lambda(\dd y) = g(y) \dd y \textnormal{ for, } y \in [0,y_0^{}] \quad \textnormal{ and } \quad  \lim_{y \to 0+} g(y) y^{1 - \alpha} = A, \label{eq:assumption A}
			\end{align}
			for some $A \in (0,\infty)$ and $\alpha \in (0,1)$; see Assumption (A) in \cite{Limic2015}. However, to keep technicalities in check, we restrict ourselves to measures with densities of the form \eqref{betadist}.
			
			In particular, a result in this spirit was obtained for the Beta coalescent in \cite{Limic2015}, where it is shown that the fluctuations of the block counting process of the Beta coalescent around its deterministic limit fulfils a functional limit theorem. More precisely, under the assumption in \eqref{eq:assumption A}, they prove that the fluctuations are given by $(2-\alpha)$ stable process of Ornstein-Uhlenbeck type, see Theorem 1.2 in \cite{Limic2015}. However, note that in their work they are able to start the coalescent with infinitely many lineages.
		\end{remark}
		
		\begin{remark}
			In general, it is well-known 
			(see~\cite[Exercise~6.2.10]{Applebaum2009}) that a time-inhomogeneous $\Rb^d$-valued SDE of the form
			\begin{equation*}
				\dd Y(t) = b \big ( t, Y(t)    \big ) \dd t
				+ \sigma \big (t, Y(t-)   \big ) \dd B(t)
				+ \int_{|x| < c} F\big ( t,Y(t-),x \big ) \, \widetilde N (\dd t, \dd x) 
				+ \int_{|x| \geqslant c} F \big ( t,Y(t-),x \big ) \, N (\dd t, \dd x) 
			\end{equation*}
			has a unique solution if the coefficients satisfy the Lipschitz-condition 
			\begin{equation*}
				\begin{split}
					&| b(t,y_1^{}) - b(t,y_2^{}) |^2 
					+
					\| a(t,y_1^{},y_1^{}) - 2 a(t,y_1^{},y_2^{}) + a(t,y_2^{},y_2^{}) \| \\
					& \quad +
					\int_{|x| < c} |F(t,y_1^{},x) - F(t,y_2^{},x)|^2 \, \nu (\dd x)
					\leqslant K_1 (t) |y_1^{} - y_2^{}|^2.
				\end{split}
			\end{equation*}
			and the growth condition
			\begin{equation*}
				| b(t,y) |^2 + \| a(t,y,y) \| + \int_{|x| < c}
				|F(t,y,x)|^2 \, \nu (\dd x) \leqslant K_2 (t) (1 + |y|^2).
			\end{equation*}
			Here, $a(t,y_1^{},y_2^{}) = \sigma(t,y_1^{}) \sigma(t,y_2^{})^T$ 
			and $K_1,K_2$ are locally bounded and measurable. 
			
			Since there is no diffusion term in our SDE for $\bU = (\bU_t)_{t \geqslant 0}$ in Theorem~\ref{thm:fclt} and the martingale does not depend on 
			$\bU$, it suffices to check that
			\begin{equation*}
				| DF(\cs_t^{}) (y) | \leqslant K_1 (t) |y|
			\end{equation*}
			where $DF(\cs_t^{})$ denotes the Jacobian of $F$ in 
			$\cs_t^{}$. But this is obvious because $F$ is a polynomial (and therefore smooth), restricted to a compact set whence its Jacobian is bounded. 
			
		\end{remark}
		
		\begin{remark}\label{rem:Bloznelis}
			In the literature dealing with subgraph counts, stable limits only appeared recently \cite{Bloznelis2024}, whereas normal 
			limits and normal functional limit theorems have been observed for the Bernoulli random graph case for example in \cite{Janson1990}. In particular, in \cite{Rucinski1988} a Bernoulli random graph is considered and it is shown that under suitable conditions the subgraph counts are either asymptotically normal or Poisson. In \cite{Janson1990} a dynamic Bernoulli random graph is considered, where edges appear randomly as time passes. There, a functional limit theorem for the subgraph count is proved, where the limiting process is Gaussian. In contrast, our block size spectrum for the beta coalescent with momentum
			does not exhibit Gaussian fluctuations and instead exhibits stable fluctuations (Theorem \ref{thm:fclt}).
			
			In \cite{Bloznelis2024} the authors consider a superposition of Bernoulli random graphs to construct a random graph with power-law degree distribution. They derive a stable limit theorem for the number of subgraph counts in this random graph. Even though our model and theirs are not identical, they similarly consider a graph which is constructed as a superposition of Bernoulli random graphs. In contrast to our work, their setting is entirely static. On the other hand, they proved a stable limit theorem even for cases where the edge probabilities in the superposed Bernoulli random graphs are strictly smaller than $1$. 
			
			Their observation, together with ours, raises possible questions for further work. In particular, does Theorem \ref{thm:fclt} also hold if one considers subgraph counts instead of the block size spectrum? Furthermore, so far we have only covered the case where the intensity $\mu(\dd t,\dd u, \dd q)$ of the driving Poisson point process is concentrated on $[0,\infty) \times (0,1] \times\{1\}$, i.e. the edge probability $q$ is always $1$.
			Does Theorem \ref{thm:fclt} still apply for more general measures for the edge probability $q$? 
		\end{remark}

		\section{Proofs}\label{sect:proofs}
		
		\subsection{Integral/Poisson representation} \label{subsec:poissonrep}
		The main device that is used in the proofs of Theorems~\ref{thm:lln}, \ref{thm:llnl2} and~\ref{thm:fclt} is a representation of the normalised block size spectrum $\bC^n$ in terms of an integral equation with respect to a Poisson point process similar to $N$ (see \eqref{eq:definition_of_N}), but augmented by additional information regarding which blocks are affected by each merger. For this, we let $N^E$ (the `$E$' stands for `extended') be a Poisson point process on
		$[0,\infty) \times [0,1] \times [0,1]^\Nb$
		with intensity
		\begin{equation*}
			\mu^E (\dd t, \dd u, \dd \bx) \defeq \dd t \frac{\Lambda(\dd u)}{u^2} \dd \bx,
		\end{equation*}
		where $\dd \bx$ is the uniform distribution on $[0,1]^\Nb$. Note that we may construct $N^E$ by first constructing the process $N$ as in
		Eq.~\eqref{poissonintensity}, and then sampling, independently for each atom $(t,u)$, a third component $\bx$ as a realisation of a random variable $\bX = (X_1,X_2,\ldots)$ where $X_i^{}, i \in \Nb$ are independent random variables, uniformly distributed on $[0,1]$. 
		
		Keeping with our habit of referring to atoms $(t,u)$ of $N$ as $u$-mergers, we refer to an atom $(t,u,\bx)$ of $N^E$ as a $(u,\bx)$-merger.  
		The idea is to index (over $\Nb$) the blocks of $\Pi^{n}$ and their elements in such a way that upon an $(u,\bx)$-merger, the $i$-th vertex participates in a merger if and only if $x_i^{} \leqslant u$. 
		
		Let us be a bit more precise. By exchangeability, we can index the blocks of $\Pi^n$ without loss of generality so that, whenever $\bC_t^n = \bc$, say, 
		$\pi = \Pi^{n}_{tn^{\alpha-1}}$ has the following form.
		\begin{itemize}
			\item
			The $n c_1^{}$ singleton blocks of $\pi$ are 
			$\{1\},\{2\},\ldots,\{nc_1^{} \}$.
			\item
			The $n c_2^{}$ blocks of  $\pi$ of size $2$ are 
			$\{nc_1^{} + 1, nc_1^{} + 2\}, \{nc_1^{} + 3, nc_1^{} + 4 \}, \ldots, \{ n c_1^{} + 2 nc_2^{} - 1, nc_1^{} + 2 nc_2^{}   \}$.
			\item
			In general, the $k$-th block of size $i$ is given by $I_{i,k} \defeq \{S_i + (k-1)i + 1,\ldots,S_i + ki \}$ with
			$S_i \defeq nc_1^{} + 2nc_2^{} + \ldots + (i-1)nc_{i-1}^{}$.
		\end{itemize} 
		Consequently, for any $1 \leqslant i \leqslant d$ and $1 \leqslant k \leqslant nc_i^{}$, 
		\begin{equation*}
			B_{i,k} \defeq \big \{ (u,\bx) \in [0,1] \times [0,1]^\Nb : \exists j \in I_{i,k} \textnormal{ s.t. } x_j^{} \leqslant u         \big \}
		\end{equation*}
		is the set of all $(u,\bx)$ for which the $k$-th block of size $i$ is marked by an  $(u,\bx)$-merger. Denoting for integers 
		$a \leqslant b$ the discrete interval $\{a,a+1,\ldots,b\}$ by
		$[a:b]$,
		we also define 
		\begin{equation*}
			B_{\geqslant} \defeq \big \{ (u,\bx) \in [0,1] \times [0,1]^\Nb : \exists \, j \in [ \|\bc\|n + 1 : n ]  \textnormal{ s.t. } x_j^{} \leqslant u                                  \big \},
		\end{equation*}
		the set of all $(u,\bx)$ for which some block with more than $d$ vertices is marked. Note that $B_{j,q}$ and $B_\geqslant$ all depend implicitly on $\bc$ as well as $n$.

		The advantage of this augmented representation is that, having fixed a realisation of $N^E$, there is no additional randomness; each $(u,\bx)$-merger induces a unique 
		(possibly trivial) transition of $\bC^n$. For $1 \leqslant i \leqslant d$, we will write
		$f_i^{n,+} (\bc,u,\bx)$ and $f_i^{n,-} (\bc,u,\bx)$ for the normalised number of blocks of size $i$ that are gained and lost upon an $(u,\bx)$-merger. 
		
		Clearly, every marked block of size $i$ is lost, except when no other vertex outside of that block is coloured. Therefore,
		\begin{equation} \label{fminusdef}
			f_i^{n,-} (\bc,u,\bx) = \frac{1}{n} \sum_{k = 1}^{nc_i^{}} \1_{B_{i,k}} (u,\bx) \Big ( 1 - \prod_{r \in [n] \setminus I_{i,k} } \1_{x_r^{} > u}    \Big ).
		\end{equation}
		Notices that whenever the term in brackets is $0$ for some 
		$k \in [1:n c_i^{}]$, the entire sum vanishes. 
		
		On the other hand, we gain a block of size $i$ whenever, for some $\bl \in \Ps_{2+}(i)$, exactly $\ell_j^{}$ blocks of size $j$ are marked, and no other vertex participates. Hence,
		\begin{equation} \label{fplusdef}
			f_i^{n,+} (\bc,u,\bx) = \frac{1}{n} \sum_{\bl \in \Ps_{2+} (i)}  \1_{B_{\geqslant}^c} (u,\bx)     \prod_{j = 1}^d \Big ( \sum_{\substack{K \subseteq [nc_j^{}] \\ |K| = \ell_j^{} } }   
			\prod_{q \in K} \1_{B_{j,q}} (u,\bx) \prod_{r \in [nc_j^{}] \setminus K} \1_{B_{j,r}^c} (u,\bx) \Big ).
		\end{equation}
		
		
		Finally, we also write $$f_i^{n}(\bc,u,\bx) \defeq f_i^{n,+} (\bc,u,\bx) - f_i^{n,-} (\bc,u,\bx),$$ for the net change in the (normalised) number of blocks of size $i$ caused by a 
		$(u,\bx)$-merger. To account for the time change, we write $N_n^E$ for the image of $N^E$ under the map $(t,u,\bx) \mapsto (t n^{1 - \alpha}, u, \bx)$. 
		Note that by the Poisson mapping theorem (see, for example, Proposition 11.2 in \cite{Privault2018}), $N_n^E$ is a Poisson point process with intensity 
		\begin{equation*}
			\mu_n^E (\dd t, \dd u, \dd \bx) \defeq n^{\alpha -1} \mu^E (\dd t, \dd u, \dd \bx) = n^{\alpha -1} \dd t \frac{\Lambda(\dd u)}{u^2} \dd \bx. 
		\end{equation*}
		
		With this, we now have for all $n \in \Nb_+$, $t \geqslant 0$ and $1 \leqslant i \leqslant d$, 
		\begin{equation*}
			C_{t,i}^{n} =  \delta_{i,1} + \int_{[0,t) \times [0,1] \times [0,1]^\Nb} f_i^{n} \big ( \bC_{s-}^{n}, u, \bx   \big ) N_n^E (\dd s, \dd u, \dd \bx),
		\end{equation*}
		recalling the initial condition $C_{0,i}^n = \delta_{i,1}$.
		By decomposing $N_n^E$ into the compensated Poisson point process $\widetilde N_n^E$ and its intensity $\mu_n^E$, we arrive at 
		\begin{equation} \label{decomp}
			C_{t,i}^n = \delta_{i,1} + \int_{[0,t)} F_i^n  (\bC_{s}^n) \dd s + \widehat M_{t,i}^n,
		\end{equation}
		where, for all $\bc \in E^n$, 
		\begin{equation} \label{Fdef}
			F_i^n (\bc) \defeq n^{\alpha - 1} \int_{[0,1] \times [0,1]^\Nb} f_i^{n} \big ( \bc , u, \bx   \big ) \,
			\frac{\Lambda ( \dd u)}{u^2} \dd \bx
		\end{equation}
		and
		\begin{equation} \label{mhatdef}
			\widehat M_{t,i}^n \defeq  \int_{[0,t) \times [0,1] \times [0,1]^\Nb} f_i^{n} \big ( \bC_{s-}^{n}, u, \bx   \big ) \widetilde N_n^E (\dd s, \dd u, \dd \bx).
		\end{equation}
		We also write $\widehat \bM_t^n \defeq (\widehat M_{t,1}^n,\ldots,\widehat M_{t,d}^n)$ and note that $\widehat \bM^n \defeq (\widehat \bM_t^n)_{t \geqslant 0}$ is a martingale with respect to the filtration generated
		by $N_n^E$.

		\subsection{Structure of the proof / Heuristics}
		Before diving into the computations, we start by outlining the structure of our arguments. 
		In order to prove the law of large numbers, we will show that the martingale part $\widehat \bM^n_t$ vanishes as $n \to \infty$.
		We will also show (see Lemma~\ref{lem:uniformapproximation}), that the (uniform) limit of $F_i^n$ in Eq.~\eqref{Fdef} is given by 
		$F_i^{}$ in Theorem~\ref{thm:lln}. Taking the limit on both sides of 
		Eq.~\eqref{decomp} and exchanging the limit with the integral, we expect $\cs_{t,i}^{} = \lim_{n \to \infty} C_{t,i}^n$ to satisfy the integral equation
		\begin{equation*}
			\cs_{t,i}^{} = \int_0^t F_i(\cs_t^{}) \dd s
		\end{equation*}
		or, equivalently, the ordinary differential equation
		\begin{equation*}
			\frac{\dd }{\dd t} \cs_{t,i}^{} = F_i (\cs_t^{}).
		\end{equation*}
		Controlling the $L^2$ norm of the martingale part via It\^{o} isometry will lead to the $L^2$ version of the law of large numbers in Theorem~\ref{thm:llnl2}.
		
		In order to obtain the functional limit theorem (FLT), we need a finer understanding of the asymptotics of the martingale 
		part. To this end, it is crucial to observe that due to slowing down time by a factor of $n^{\alpha -1}$, we will only see $u$-mergers for very 
		small $u$ and may neglect effects that are of higher order in $u$. In particular, recall that the probability that any given block of size $i$ is marked (and therefore lost) during an $u$-merger is
		$i u + O(u^2)$ (see Eq.~\eqref{markingprobability}) and the probability that such a block is gained is of order $u^2$ (because at least two smaller blocks need to be marked) and thus negligible.
		We therefore expect the net change in the (normalised) number of blocks of size $i$, due to a $u$-merger at rescaled time $t$, to be roughly $-iu \cs_{t,i}^{}$, which suggests the 
		approximation
		\begin{equation} \label{integrandapproximated}
			\widehat M_{t,i}^n \approx M_{t,i}^n \defeq - i \int_{[0,t) \times [0,1]} u \cs_{s,i}^{} \widetilde N_n(\dd s, \dd u)
			=
			- i \int_{[0,t) \times [0,1] \times [0,1]^\Nb} u \cs_{s,i}^{} \widetilde N_n^E (\dd s, \dd u, \dd \bx ),
		\end{equation} 
		where $\widetilde N_n$ is a compensated Poisson point process on $[0,\infty) \times (0,1)$ with intensity
		\begin{equation*}
			\mu_n^{} (\dd t, \dd u) \defeq n^{\alpha-1} \mu(\dd t, \dd u) \defeq n^{\alpha-1} \dd t \frac{\Lambda(\dd u)}{u^2}.
		\end{equation*}
		Next, we investigate the asymptotics of the integral on the right-hand side of Eq.~\eqref{integrandapproximated} and give a heuristic for the scaling 
		$\sigma_n^{}$ in Theorem~\ref{thm:fclt}. First, note that due to the presence of the factor $n^{\alpha -1}$ in the density $\mu_n^{}$, $\bM^n$ will vanish as 
		$n \to \infty$, hence we need to scale it up by a factor $\sigma_n^{}$ to obtain a nontrivial limit. We see that
		\begin{equation*}
			\sigma_n^{} M_{t,i}^n = -i \int_{[0,t) \times [0,1]} (\sigma_n^{} u) \cs_{s,i}^{} \widetilde N_n (\dd s, \dd u) = 
			-i \int_{[0,t) \times [0,\sigma_n^{}]} u \cs_{s,i}^{} \widetilde N_n^\ast (\dd s, \dd u),
		\end{equation*}
		where $N_n^\ast$ is the image of $N_n$ under the map $(t,u) \mapsto (t, \sigma_n^{} u)$, which is by the Poisson mapping theorem
		(see again~\cite{Privault2018}) a Poisson point process with intensity 
		\begin{equation*}
			n^{\alpha -1} \dd t  (u/\sigma_n^{})^{\alpha - 3} (1 - u/\sigma_n^{})^{\beta -1} \dd (u / \sigma_n^{}) = 
			\sigma_n^{2 - \alpha} n^{\alpha - 1}  \dd t u^{\alpha -3} (1- u/\sigma_n^{})^{\beta -1} \dd u,
		\end{equation*}
		which converges to the intensity of $\Ns$ in Theorem~\ref{thm:fclt} upon choosing $\sigma_n^{} = n^{(1-\alpha)/(2-\alpha)}$.
		
		\subsection{Rigorous argument}
		To prepare our proof of Theorem~\ref{thm:lln}, we first prove that the functions $F_i^n$ in Eq.~\eqref{Fdef} converge uniformly to $F_i^{}$ given in Theorem~\ref{thm:lln}. In view of later applications in the proof of 
		Theorem~\ref{thm:fclt}, we need 
		some quantitative control.
		\begin{lemma} \label{lem:uniformapproximation}
			For $1 \leqslant i \leqslant d$, we have
			\begin{equation*}
				\sup_{\bc \in E^n} | F_i^n (\bc) - F_i^{} (\bc) | = O(n^{\alpha -1}).
			\end{equation*}
		\end{lemma}
		Here and in the following, $O(\cdot)$ refers to the limit $n\to \infty$ unless mentioned otherwise. Additionally, $K>0$ will denote some constant which precise value might change from line to line.
		\begin{proof}
			Recall the definition of $f_i^{n,\pm}$ in Eqs.~\eqref{fminusdef} and \eqref{fplusdef}. Define accordingly, mimicking Eq.~\eqref{Fdef},
			\begin{equation*}
				F_i^{n,\pm}(\bc) \defeq n^{\alpha -1} \int_{[0,1] \times [0,1]^\Nb} f_i^{n,\pm} (\bc,u,\bx) \frac{\Lambda(\dd u)}{u^2} \dd \bx.
			\end{equation*}
			Let also
			\begin{equation*}
				F_i^+ (\bc) \defeq \sum_{\bl \in \Ps_{2+}(i)} \prod_{j=1}^d \frac{(jc_j^{})^{\ell_j^{}}}{\ell_j^{} !} \Gamma(\alpha + |\bl| - 2) 
			\end{equation*}
			and
			\begin{equation*}
				F_i^-(\bc) \defeq \frac{\Gamma(\alpha)}{1-\alpha} i c_i^{}.
			\end{equation*}
			We will separately show that 
			\begin{equation*}
				\sup_{\bc \in E^n} | F_i^{n,\pm} (\bc) - F_i^{\pm} (\bc) | = O(n^{\alpha -1}).
			\end{equation*}
			Using the definition of $f_i^{n,+}$ (see \eqref{fplusdef}) and the fact that $\1_{B_{j,q}}$ depends only on $x_i^{}$ with $i \in I_{j,q}$, and $I_{j,q} \cap I_{h,r} = \varnothing$ for $j \neq h$, and that $\1_{B_\geqslant}$ (and hence $\1_{B_\geqslant^c}$ ) only depends on $[\|\bc\|n + 1,\ldots,n]$, which is disjoint from all $I_{j,q}$, we get for $i \in [2:d]$, recalling the definition of $p_j^{} (u)$ in
			Eq.~\eqref{markingprobability}
			\begin{equation} \label{saiodfc}
				\begin{split}
					&F_i^{n,+}(\bc)	= n^{\alpha-1} \int_{[0,1] \times [0,1]^\Nb} f_i^{n,+} (\bc,u,\bx) \frac{\Lambda(\dd u)}{u^2} \dd \bx \\ 
					&\quad  = n^{\alpha - 2} \sum_{\bl \in \Ps_{2+}(i)} \int_0^1 (1-u)^{n - n\|\bc\|} \prod_{j=1}^d \left[ \binom{nc_j^{}}{\ell_j^{}} p_j^{}(u)^{\ell_j^{}} (1-u)^{j n c_j -j  \ell_j} \right]
					u^{\alpha-3}(1-u)^{\beta -1} \dd u\\
					&\quad  = n^{\alpha -2} \sum_{\bl \in \Ps_{2+}(i)} \binom{n\bc}{\bl} \int_0^1 \left[ \prod_{j=1}^d p_j^{}(u)^{\ell_j^{}} \right] u^{\alpha - 3} (1-u)^{n - \|\bl\| + \beta -1} \dd u,
				\end{split}
			\end{equation}
			where for the last step, we used the notation
			\begin{equation*}
				\binom{n\bc}{\bl} \defeq \prod_{j=1}^d \binom{nc_j^{}}{\ell_j^{}}.
			\end{equation*}
			Next, we deal with the $u$-integration. For this, we need to evaluate for all $\bl \in \Ps_{2+}(i)$ the limit  of the integral
			\begin{equation*}
				\int_0^1 P_{\bl}(u) u^{\alpha-3} (1-u)^{n - \| \bl \| + \beta -1} \dd u
			\end{equation*}
			as $n\to \infty$. Here,
			$P_{\bl}(u) \defeq \prod_{j=1}^d p_j^{} (u)^{\ell_j^{}} = u^{|\bl|} \prod_{j=1}^d j^{\ell_j^{}} + O(u^{|\bl| + 1 })$, as $u\to 0$.  By Lemma~\ref{lem:integralestimate1}, we have
			\begin{equation*}
				\int_0^1 P_{\bl}(u) u^{\alpha-3} (1-u)^{n - \| \bl \| + \beta -1} \dd u = n^{2 - \alpha - |\bl|} \prod_{j=1}^d j^{\ell_j^{}} \Gamma(\alpha -2 + |\bl|) 
				+ O(n^{1 - \alpha - |\bl|}).
			\end{equation*}
			Inserting this into Eq.~\eqref{saiodfc} and noting that
			\begin{equation*}
				\binom{n\bc}{\bl} = \frac{n^{|\bl|}}{\bl !} \prod_{j=1}^d c_j^{\ell_j} + O(n^{|\bl| - 1}),
			\end{equation*}
			uniformly in $\bc$ with the convention $\bl ! \eqdef \prod_{j=1}^d \ell_j!$, we see that
			\begin{equation*}
				\begin{split}
					F_i^{n,+}(\bc)&= \sum_{\bl \in \Ps_{2+}(i)} \binom{n\bc}{\bl} \Big ( n^{-|\bl|} \Gamma(\alpha -2 + |\bl|) \prod_{j=1}^d j^{\ell_j^{}} 
					+ O(n^{-|\bl| - 1}) \Big ) \\
					& \quad = \sum_{\bl \in \Ps_{2+}(i)}  \prod_{j=1}^d \frac{(jc_j^{})^{\ell_j^{}}}{\ell_j^{}!} \Gamma(\alpha -2 + |\bl|) + O(n^{-1}) \\
					& \quad = F_i^+(\bc) + O(n^{-1}),
				\end{split}
			\end{equation*}
			where the error term is uniform in $\bc$. 
			
			Next, we show the convergence of $F_i^{n,-}$. Proceeding as before, we have
			\begin{equation*}
				\begin{split}
					F_i^{n,-} (\bc) &= n^{\alpha - 1} \int_{[0,1] \times [0,1]^\Nb} f_i^{n,-} (\bc,u,\bx) \frac{\Lambda(\dd u)}{u^2} \dd \bx \\
					&  = n^{\alpha - 1} c_i^{} \int_0^1 p_i^{} (u) \big (1 - (1-u)^{n-i}   \big ) u^{\alpha -3} (1-u)^{\beta - 1} \dd u \\
					&  = n^{\alpha - 1} i c_i^{} \int_0^1 u^{\alpha - 2} (1 - u)^{\beta - 1}  \big (1 - (1-u)^{n-i}   \big )  \dd u + O(n^{\alpha -1}),
				\end{split}
			\end{equation*}
			where we used in the second step that $p_i^{}(u) = i u + O(u^2)$.
			From Lemma~\ref{lem:integralestimate2} we see that
			\begin{equation*}
				\int_0^1 u^{\alpha - 2} (1 - u)^{\beta - 1}  \big (1 - (1-u)^{n-i}   \big ) = n^{1-\alpha} \frac{\Gamma(\alpha)}{1 - \alpha} + O(1),
			\end{equation*}
			which finishes the proof.
		\end{proof}
		
		Next, we will prove an a-priori estimate for the martingales $\widehat M_i^n$.
		
		\begin{lemma} \label{lem:apriorimartingaleestimate}
			For all $1 \leqslant i \leqslant d$ and $T > 0$, 
			\begin{equation*}
				\EE{(\widehat M_{T,i}^n)^2} = O(n^{\alpha - 1}).
			\end{equation*}
		\end{lemma}
		\begin{proof}
			We define, decomposing $f_i^n$ in Eq.~\eqref{mhatdef} into $f_i^{n,\pm}$,
			\begin{equation*}
				\widehat M_{T,i}^{n,\pm} \defeq  \int_{[0,T) \times [0,1] \times [0,1]^\Nb} f_i^{n,\pm} \big ( \bC_{s-}^{n}, u, \bx   \big ) \widetilde N_n^E (\dd s, \dd u, \dd \bx).
			\end{equation*}
			By Ito isometry, we have
			\begin{equation} \label{Mhatplussquared}
				\EE{(\widehat M_{T,i}^{n,+})^2} 
				=
				\int_{ [0,T) \times [0,1] \times [0,1]^\Nb } f_i^{n,+}(\bC_{s}^n,u,\bx)^2 \, \mu_n^{} (\dd s, \dd u) \dd \bx.
			\end{equation}
			To evaluate this, note that for $\bX \sim \textnormal{unif}([0,1]^\Nb)$, $n f_i^{n,+}(\bC_{s}^n,u,\bX)$ conditional on $\bC_{s}^n$ is a Bernoulli random variable with success probability
			\begin{equation*}
				\sum_{\bl \in \Ps_{2+}(i)} (1-u)^{n-n \| \bC^n_{s} \| } \prod_{j=1}^d \binom{n C^n_{s,j}}{\ell_j^{}} p_j^{} (u)^{\ell_j} (1-u)^{j (n  C^n_{s,j} - \ell_j^{})} .
			\end{equation*}
			Thus, the integral on the right-hand side of Eq.~\eqref{Mhatplussquared} evaluates to 
			\begin{equation*}
				\begin{split}
					& \quad  
					n^{\alpha -3} \sum_{\bl \in \Ps_{2+}(i)}  \int_{[0,T) \times [0,1]}  \binom{n\bC^n_{s}}{\bl} P_{\bl} (u) u^{\alpha -3}  (1-u)^{n - \|\bl\| + \beta -1} \dd u \dd s 
					\\
					& .\quad \leqslant K \sum_{\bl \in \Ps_{2+}(i)} n^{\alpha - 3 + |\bl|} \int_0^1 P_{\bl} (u) u^{\alpha -3} (1-u)^{n- \|\bl\| + \beta -1} \dd u,
				\end{split}
			\end{equation*}
			with $P_{\bl}(u) = \prod_{j=1}^d j^{\ell_j} u^{|\bl|} + O(u^{|\bl| + 1})$ as $u\to 0$ (see the proof of Lemma~\ref{lem:uniformapproximation}) and some uniform constant $K$. By Lemma~\ref{lem:integralestimate1}, we have the estimate
			\begin{equation*}
				\int_0^1 P_{\bl} (u) u^{\alpha -3} (1-u)^{n- \|\bl\| + \beta -1} \dd u = O(n^{2 - \alpha - |\bl|}).
			\end{equation*}
			We have thus shown that 
			\begin{equation} \label{Mhatplusdone}
				\EE{(\widehat M_{T,i}^{n,+})^2}  = O(n^{-1}) = O(n^{\alpha -1}).
			\end{equation}
			We now turn to estimating $\EE{(\widehat M_{T,i}^{n,-})^2}$. Again by Ito isometry, 
			\begin{equation*}
				\EE{(\widehat M_{T,i}^{n,-})^2} 
				=\int_{ [0,T) \times [0,1] \times [0,1]^\Nb } f_i^{n,-}(\bC_s^n,u,\bx)^2 \, \mu_n^{} (\dd s, \dd u) \dd \bx.
			\end{equation*}
			Recall the definition of $f_i^{n,-}$ in Eq.~\eqref{fminusdef}. Bounding the number of marked blocks of size $i$ by the number of coloured vertices we see that for $\bX \sim \textnormal{unif}([0,1]^\Nb)$ and fixed $u \in [0,1]$ and
			$\bc \in E^n$, there is $B \sim \textnormal{Binomial}(n,u )$ s.t.  $n f_i^{n,-}(\bc,u,\bX)$ is dominated by $B \1_{B \geqslant 2}$ for all $\bc$
			so that we get the bound
			\begin{equation*}
				\EE{(\widehat M_{T,i}^{n,-})^2}  \leqslant K n^{\alpha - 3} \int_0^1 \EE{B^2 \1_{B \geqslant 2}^{}} u^{\alpha -3} (1-u)^{\beta -1} \dd u,
			\end{equation*}
			for some uniform constant $K > 0$.
			We have the bound
			\begin{equation*}
				\EE{B^2 \1_{B \geqslant 2}^{}}
				= \EE{B^2 - \1_{B = 1}^{}}
				= n^2 u^2 + nu(1-u) - nu(1-u)^{n-1}
				\leqslant n^2 u^2 + n u \big (1 - ( 1 - u )^{n}         \big ).
			\end{equation*}
			Clearly,
			\begin{equation*}
				\int_0^1 u^{\alpha - 1} (1 - u )^{\beta - 1} \dd u < \infty,
			\end{equation*}
			and by Lemma~\ref{lem:integralestimate2},
			\begin{equation*}
				\int_0^1 u^{\alpha - 2}  \big (1 - ( 1 - u )^{n}         \big ) (1 - u)^{\beta - 1} \dd u
				=
				O(n^{1 - \alpha}).
			\end{equation*}
			Altogether, this shows that $\EE{(\widehat M_{T,i}^{n,-})^2}  = O(n^{\alpha -1})$, and together with Eq.~\eqref{Mhatplusdone}, and the elementary inequality
			$(a+b)^2 \leqslant 2a^2 + 2b^2$,
			this concludes the proof.
		\end{proof}
		
		We are now ready to prove the law of large numbers. 
		\begin{proof}[Proof of Theorem~\ref{thm:lln}] Letting $\varepsilon>0$ and $i \in \{1,\dots,d \}$, we start from the representation~\eqref{decomp} 
			\begin{equation*}
				C_{t,i}^n = \delta_{i,1} + \int_0^t F_i^n (\bC_s^n) \dd s + \widehat M_{t,i}^n.
			\end{equation*}
			By definition, $\cs$ is the solution of the integral equation
			\begin{equation*}
				\cs_{t,i}^{} = \delta_{i,1} +  \int_0^t F_i (\cs_s^{}) \dd s.
			\end{equation*}
			Thus, we have for all $t \in [0,T]$
			\begin{equation} \label{eq:qw00000üfj}
				C_{t,i}^n - \cs_{t,i}^{} = \int_0^t F_i^n (\bC_s^n) - F_i^{} (\bC_s^n) \dd s + \int_0^t F_i^{} (\bC_s^n) - F_i (\cs_s^{}) \dd s + \widehat M_{t,i}^n.
			\end{equation}
			Lemma~\ref{lem:uniformapproximation} yields a deterministic bound for the first integral.
			\begin{equation*}
				\left | \int_0^t F_i^n (\bC_s^n)   - F_i^{} (\bC_s^n) \dd s   \right | 
				\leqslant
				\int_0^t |       F_i^n (\bC_s^n)   - F_i^{} (\bC_s^n)                           | \dd s \leqslant T \sup_{\bc \in E^n} | F_i^n (\bc) - F_i^{} (\bc) | = O(n^{\alpha - 1}).
			\end{equation*}
			Setting
			\begin{equation} \label{Ddef}
				D_t^n \defeq \max_{1 \leqslant i \leqslant d} | C_{t,i}^n   - \cs_{t,i}^{}|
			\end{equation}
			and by the mean-value theorem (noting that $F_i^{}$ is smooth), we see that 
			Eq.~\eqref{eq:qw00000üfj} implies
			\begin{equation*}
				D_t^n \leqslant K \int_0^t D_s^n \dd s + \max_{1 \leqslant i \leqslant d} |\widehat M_{t,i}^n| + O(n^{\alpha - 1}).
			\end{equation*}
			For $\varepsilon' > 0$, let $A_{n,\varepsilon'}$ be the event that
			\begin{equation*}
				\max_{1 \leqslant i \leqslant d} \sup_{t \in [0,T]} | \widehat M^n_{t,i} | \leqslant \varepsilon'.
			\end{equation*}
			By Doob's inequality and Lemma~\ref{lem:apriorimartingaleestimate},
			\begin{equation} \label{eq:acbound}
				\PP{A^c_{n,\varepsilon'}} \leqslant 
				4 {(\varepsilon')}^{-2} \sum_{i=1 }^d \EE{(\widehat M_{T,i}^n)^2} = O(n^{\alpha-1}).
			\end{equation}
			Moreover, on the event $A_{n,\varepsilon'}$, we have for sufficiently large $n$ and $t \in [0,T]$ that 
			\begin{equation*}
				D_t^n \leqslant K \int_0^t D_s^n \dd s + 2 \varepsilon'
			\end{equation*}
			and Grönwall's inequality implies that
			\begin{equation} \label{eq:gronwallconsequence}
				D_t^n \leqslant 2 \varepsilon' (1 + T\ee^{T})
			\end{equation}
			for all $t \in [0,T]$.
			By the equivalence of the euclidean and maximum norms, there is a constant $K > 0$ such that
			\begin{equation*}
				\PP{\sup_{t \in [0,T]} |\bC^n_t - \cs_{t} |  > \varepsilon} \leqslant 
				\PP{\sup_{t \in [0,T]} K D_t^n > \varepsilon }. 
			\end{equation*}
			Letting
			$\varepsilon^\ast \defeq \frac{\varepsilon}{2 K (1 + T \ee^T)}$,
			we can use Eqs.~\eqref{eq:gronwallconsequence} and \eqref{eq:acbound} to finally conclude that
			\begin{equation*}
				\PP{\sup_{t \in [0,T]} K D_t^n > \varepsilon } 
				\leqslant \PP{A_{n,\varepsilon^\ast}^{c}}
				= O(n^{\alpha - 1}).
			\end{equation*}
		\end{proof}
		
		The proof of the $L^2$-version uses similar arguments.
		\begin{proof}[Proof of Theorem~\ref{thm:llnl2}]
			We start from the representation in Eq.~\eqref{decomp}.
			\begin{equation*}
				C_{t,i}^n = \delta_{i,1} + \int_0^t F_i^n (\bC_s^n) \dd s + \widehat M_{t,i}^n
			\end{equation*}
			Recalling the definition of $\cs$, we have for all $t \in [0,T]$
			\begin{equation*}
				C_{t,i}^n - \cs_{t,i}^{} = \int_0^t F_i^n (\bC_s^n) - F_i^{} (\bC_s^n) \dd s + \int_0^t F_i^{} (\bC_s^n) - F_i (\bC_s^{}) \dd s + \widehat M_{t,i}^n.
			\end{equation*}
			Bounding the first integral with the help of Lemma~\ref{lem:uniformapproximation} and using the elementary estimate \\$(a+b+c)^2~\leqslant~3(a^2+b^2+c^2)$, we see that
			\begin{equation*}
				\begin{split}
					\big ( C_{t,i}^n - \cs_{t,i}^{} \big )^2 & \leqslant  3 \Big ( \int_0^t F_i(\bC_s^n) - F_i (\bC_s^{}) \dd s       \Big )^2 + 3 (\widehat M_{t,i}^n)^2 + O(n^{\alpha -1}) \\
					& \leqslant 3 \int_0^t \big ( F_i (\bC_s^n) - F_i (\bC_s^ {})          \big )^2 \dd s + 3 (\widehat M_{t,i}^n)^2 + O(n^{\alpha -1}), 
				\end{split}
			\end{equation*}
			where the second step is an application of Jensen's inequality and the error is uniform in $t\in [0,T]$. Next, we take the maximum over all $1 \leqslant i \leqslant d$ 
			and the expectation and obtain, using the smoothness of $F_i$,
			\begin{equation*}
				\EE{ \max_{1 \leqslant i \leqslant d} \big ( C^n_{t,i} - \cs_{t,i}^{}         \big )^2            }
				\leqslant
				K \int_0^t \EE{ \max_{1 \leqslant i \leqslant d} \big ( C^n_{s,i} - \cs_{s,i}^{}  \big )^2 } \dd s + O(n^{\alpha - 1}), 
			\end{equation*}
			where we also used Lemma~\ref{lem:apriorimartingaleestimate}. The claim follows from Grönwall's inequality.
		\end{proof}
		
		\subsection{The functional limit theorem}
		
		As a first step towards the proof of the functional limit theorem~\ref{thm:fclt}, we make the approximation in Eq.~\eqref{integrandapproximated} precise.
		Recall that $\sigma_n^{} = n^{\frac{1- \alpha }{2 - \alpha}}$. 
		
		\begin{lemma}  \label{lem:replacetheintegrand}
			Let $\widehat \bM^n$ be as in Eq.~\eqref{mhatdef} and $\bM^n$ as in Eq.~\eqref{integrandapproximated}. Then, for any $T > 0$, $\varepsilon > 0$ and 
			$i \in \{1,\ldots,d\}$
			\begin{equation*}
				\PP{ \sup_{t \in [0,T]} \sigma_n^{} | \widehat M_{t,i}^n  - M_{t,i}^n | \geqslant \varepsilon } \to 0, \quad \textnormal{as } n \to \infty.
			\end{equation*}
			
		\end{lemma}
		\begin{proof}
			By Doob's inequality, it is enough to show that
			\begin{equation*}
				\sigma_n^2 \EE{ \big ( \widehat M_{T,i}^n - M_{T,i}^n  \big )^2 }  \to 0.
			\end{equation*}
			With $\widehat \bM^n = \widehat \bM^{n, +} - \widehat \bM^{n,-}$ as in the proof of Lemma~\ref{lem:apriorimartingaleestimate} and by the elementary inequality $(a + b)^2 \leqslant 2 a^2 + 2 b^2$,
			\begin{equation*}
				\EE{ \big ( \widehat M_{T,i}^n - M_{T,i}^n  \big )^2 } \leqslant 2 \EE{( \widehat M_{T,i}^{n,+})^2} + 2 \EE{ \big ( \widehat M_{T,i}^{n,-} + M_{T,i}^n           \big )^2 }. 
			\end{equation*}
			We have already shown (see Eq.~\eqref{Mhatplusdone}) that $\EE{( \widehat M_{T,i}^{n,+})^2} = O(n^{-1}) = o(\sigma_n^2)$. By Ito isometry,
			\begin{equation*}
				\begin{split}
					\EE{ \big ( \widehat M_{T,i}^{n,-} + M_{T,i}^n \big )^2 }&=
					\int_{[0,T) \times [0,1] \times [0,1]^\Nb} \EE{\big (f_i^{n,-} (\bC_s^n,u,x) - iu \cs_{s,i}^{}   \big )^2} \mu_n^{} (\dd s, \dd u) \dd \bx \\
					& \leqslant
					2 \int_{[0,T) \times [0,1] \times [0,1]^\Nb} \EE{\big (f_i^{n,-} (\bC_s^n,u,x) - iu C_{s,i}^n   \big )^2} \mu_n^{} (\dd s, \dd u) \dd \bx \\
					& \quad + 2 \int_{[0,T) \times [0,1]} \EE{\big (iu C_{s,i}^n - iu \cs_{s,i}^{}   \big )^2} \mu_n^{} (\dd s, \dd u).
				\end{split}
			\end{equation*}
			We can estimate the second integral with the help of Theorem~\ref{thm:llnl2}. For some constant $K > 0$, we have 
			
			\begin{align*}
				\int_{[0,T) \times [0,1] } \EE{\big (iu C_{s,i}^n - iu \cs_{s,i}^{}   \big )^2} \mu_n^{} (\dd s, \dd u) 
				&\leqslant \int_{[0,T) \times [0,1]}
				i^2 u^2 \sup_{0 \leqslant s < T, \, i \in [d]} 
				\EE{\big (C_{s,i}^n - \cs_{s,i}^{}  \big)^2}
				\mu_n^{} (\dd s, \dd u)  \\
				&\leqslant \int_{[0,T) \times [0,1]} K u^2 n^{\alpha - 1} \mu_n^{} (\dd s, \dd u)  \\
				&=
				K T n^{\alpha - 1}  \int_0^1 u^2 n^{\alpha - 1} u^{\alpha - 3} (1-u)^{\beta - 1}     \dd u\\
				&= O(n^{2\alpha - 2}) = o(\sigma_n^2),
			\end{align*}
			where we used Theorem~\ref{thm:llnl2} in the second step.
			To estimate the first integral, we proceed similarly as in the proof of Lemma~\ref{lem:apriorimartingaleestimate}. For fixed $u \in [0,1]$ and 
			$\bc \in E^n$, we have
			\begin{align*}
				\int_{[0,1]^\Nb} \big ( f_i^{n,-}(\bc,u,\bx) - iu c_i^{}         \big )^2 \dd \bx
				&\leqslant
				2 \int_{[0,1]^\Nb } \big ( f_i^{n,-} (\bc,u,\bx) - p_i^{} (u) c_i^{} \big )^2 \dd x \\
				&\qquad  + 2 \int_{[0,1]^\Nb} \big ( p_i^{} (u) c_i^{}  -  i u c_i^{}   \big )^2 \dd \bx.
			\end{align*}
			We bound the first integral via the following probabilistic interpretation. Recalling Eq.~\eqref{fminusdef}, we have for $\bX \sim \textnormal{unif}([0,1]^\Nb)$ the following equality in distribution.
			\begin{equation*}
				f_i^{n,-} (\bc,u,\bX)
				=
				\frac{1}{n} B \big ( \1_{B \geqslant 2}^{} + B' \1_{B = 1}^{}   \big ),
			\end{equation*}
			where $B \sim \textnormal{Binomial} \big (n c_i^{}, p_i^{} (u) \big )$ and  $B' \sim \textnormal{Ber} \big (1 - (1 - u)^{n - i c_i^{} n}     \big )$ 
			independently of $B$.
			This is because upon a $(u,\bX)$-merger, there is a number $B$ of marked blocks of size $i$. These are removed if either $B \geqslant 2$, or if $B = 1$ and there is at least one coloured vertex that is not part of a block of size $i$, which happens with probability $\big( 1 - (1 - u)^{n - i c_i^{} n}     \big )$.
			Thus,
			\begin{equation} \label{q390üt4}
				\int_{[0,1]^\Nb } \big ( f_i^{n,-} (\bc,u,\bx) - p_i^{} (u) c_i^{} \big )^2 \dd \bx  = 
				n^{-2} \EE{ \big ( B \1_{B \geqslant 2}^{} + B' \1_{B = 1}^{} - n p_i^{} (u) c_i^{} \big )^2  }.
			\end{equation}
			To evaluate this further, note that
			\begin{equation*}
				\big ( B \1_{B \geqslant 2}^{} + B' \1_{B = 1}^{} - n p_i^{} (u) c_i^{} \big )^2   
				=
				\big ( B - n p_i^{} (u) c_i^{} \big )^2 - \1_{B = 1}^{} \1_{B' = 0}^{} \big [  1 - 2 n p_i^{} (u) c_i^{}            \big ]. 
			\end{equation*}
			Inserting this into Eq.~\eqref{q390üt4}, we see that 
			\begin{equation*}
				\begin{split}
					& n^2 \int_{[0,1]^\Nb} \big ( f_i^{n,-}(\bc,u,\bx) - iu c_i^{}         \big )^2 \dd \bx \\
					& \quad =
					\Var{B} + \big ( 2 n p_i^{} (u) c_i^{} - 1  \big ) \PP{B = 1} \PP{B' = 0}  \\ 
					& \quad = n c_i^{} p_i^{} (u) \big ( 1 - p_i^{} (u)  \big ) + \big ( 2 n p_i^{} (u) c_i^{} - 1   \big ) \big ( n c_i^{} p_i^{} (u) \big ( 1 - p_i^{} (u)  \big )^{nc_i^{} - 1}      \big ) ( 1 -u )^{n - n c_i^{} i} \\
					& \quad = 
					n c_i^{} p_i^{} (u) \big ( 1 - (1 - u)^{n - i}       \big ) - n c_i^{} p_i^{} (u)^2 + 2 n^2 c_i^2 p_i^{} (u)^2 (1 - u)^{n - i}.
				\end{split}
			\end{equation*}
			We separately multiply each of the terms in the last line with $n^{\alpha - 3} u^{\alpha - 3} (1 - u)^{\beta - 1}$ and integrate with respect to $u$. We get, making use of the fact $p_i(u)\leqslant i u$
			\begin{align*}
				n^{\alpha - 3}  n c_i^{} \int_0^1 p_i^{} (u) \big ( 1 - (1 - u)^{n - i}           \big ) u^{\alpha - 3} (1 - u)^{\beta - 1} \dd u
				&\leqslant n^{\alpha - 2} \Big ( i n ^{1 - \alpha} \frac{\Gamma(\alpha)}{1 - \alpha}   + O(1) \Big )\\
				&= O(n^{-1}) = o(\sigma_n^{-2}),
			\end{align*}
			by Lemma~\ref{lem:integralestimate2}. By Lemma~\ref{lem:integralestimate1}, we get
			\begin{align*}
				n^{\alpha - 3} n^2 c_i^2 \int_0^1 p_i^{} (u)^2 (1 - u)^{n - i} u^{\alpha - 3} (1-u)^{\beta - 1} \dd u &\leqslant
				n^{\alpha - 1} \big ( i^2 n^{-\alpha} \Gamma(\alpha) + O(n^{-\alpha - 1})                  \big )
				\\
				&= O(n^{-1}) = o(\sigma_n^{-2}).
			\end{align*}
			Moreover, the integration of the middle term gives
			\begin{align*}
				n^{\alpha-3} n \int_0^1 c_i^{} p_i^{} (u)^2 u^{\alpha-3} (1-u)^{\beta-1}  \dd u &\leqslant n^{\alpha-2} c_i i^2 \int_0^1 u^{\alpha-1} (1-u)^{\beta-1} \dd u\\
				&= O(n^{\alpha-2})=o(\sigma_n^{-2}).
			\end{align*}
			All of these errors are uniform in $\bc$, and the proof is thus finished.
		\end{proof}
		
		Our final ingredient for the proof of Theorem~\ref{thm:fclt} is the convergence of $\sigma_n^{} M_{t,i}^n$ to $M_{t,i}^{}$. In the next lemma we consider an appropriate coupling of the Poisson point processes $\Ns$ and $N_n$ for all $n \in \Nb$.
		
		\begin{lemma} \label{lem:replacingthenoise}
			For any $T > 0$, $\varepsilon > 0$ and $i \in \{1, \ldots, d\}$, there exists a coupling of $\Ns$ and $(N_n)_{n\geqslant 1}$ such that
			\begin{equation}\label{eq:sup}
				\PP{\sup_{t \in [0,T]} | \sigma_n^{} M_{t,i}^n  - M_{t,i} |  \geqslant \varepsilon} \to 0, \quad \textnormal{as } n \to \infty.
			\end{equation}
		\end{lemma}
		\begin{proof}
			We will give a construction of $(N_n)_{n \geqslant 1}$ in terms of $\Ns$ and some additional randomness. The details of this construction will depend on the sign of $\beta - 1$. We start with the case $\beta < 1$. Given (a realisation of) $\Ns$, we define, for each $n \in \Nb$, a Poisson point process $N_n^\Delta$ 
			on $[0,\infty] \times [0, \sigma_n^{})$ with intensity
			\begin{equation*}
				\bigg (\Big ( 1 - \frac{u}{\sigma_n^{}}    \Big )^{\beta - 1} - 1 \bigg )  u^{\alpha - 3} \1_{(0,\sigma_n^{})} (u) \dd s \dd u
				=
				\left | \Big ( 1 - \frac{u}{\sigma_n^{}}    \Big )^{\beta - 1} - 1  \right | u^{\alpha-3}  \1_{(0,\sigma_n^{})} (u)  \dd s \dd u
				\eqdef \mu_n^{\Delta} (\dd s \dd u),
			\end{equation*}
			independently of $\Ns$. Then, we set
			\begin{equation*}
				N_n^\ast \defeq \Ns \cap \big ( [0,\infty) \times [0,\sigma_n^{}]    \big ) + N_n^\Delta,
			\end{equation*}
			where $+$ stands for the superposition of Poisson point processes.
			We define $N_n^{}$ as the image of $N_n^\ast$ under the map $(t,u) \mapsto (t, u/\sigma_n^{})$. By the Poisson mapping theorem, $N_n$ indeed has the desired distribution, i.e. it is a Poisson point process on $[0,\infty) \times (0,1)$ with intensity $\mu_n^{} (\dd t, \dd u)$.
			
			In the case $\beta \geqslant 1$, we construct $N_n^\Delta$ as 
			a thinning of $\Ns$; we keep each point 
			$(t,u) \in \Ns \cap \big ( [0,\infty) \times [0,\sigma_n^{})   \big ) $ with probability 
			\begin{equation*}
				1 - \Big ( 1 - \frac{u}{\sigma_n^{}}    \Big )^{\beta - 1}
			\end{equation*}
			and disregard the rest.
			Then, $N_n^\Delta$ is a Poisson point process with intensity
			\begin{equation*}
				\bigg (1 -\Big ( 1 - \frac{u}{\sigma_n^{}}    \Big )^{\beta - 1} \bigg )  u^{\alpha - 3} \1_{(0,\sigma_n^{})} (u) \dd s \dd u
				=
				\mu_n^\Delta (\dd s \dd u)
			\end{equation*}
			as above
			and we define 
			$N_n^\ast \defeq \Ns \cap \big ( [0,\infty) \times [0,\sigma_n^{})    \big )- N_n^\Delta$, where $-$ denotes the set difference of Poisson point processes.
			Again $N_n^{}$ is the image of $N_n^\ast$ under
			$(t,u) \mapsto (t,u/\sigma_n^{})$. Also in this case, it is easily checked that $N_n^{}$ has the desired distribution.
			
			With $\Ns_n' \defeq \Ns \cap \big ( [0,\infty) \times [\sigma_n^{}, \infty)   \big )$, we have in any case the decomposition
			\begin{equation*}
				\Ns = N_n^\ast + \textnormal{sgn} (\beta - 1) N_n^\Delta + \Ns_n'.
			\end{equation*}
			Applying this to the definition of $M_{t,i}$, we get 
			\begin{equation*}
				M_{t,i}^{} = M_{t,i}^{\ast,n} + \textnormal{sgn} (\beta - 1) M_{t,i}^{\Delta,n} + M_{t,i}^{\prime,n},
			\end{equation*}
			where
			\begin{align*}
				M_{t,i}^{\ast,n} &\defeq -i \int_{[0,t) \times (0,\sigma_n^{})} u \cs_{s,i}^{} \widetilde N_n^{\ast} (\dd s, \dd u) \\
				M_{t,i}^{\Delta,n} &\defeq  -i \int_{[0,t) \times (0,\sigma_n^{})} u \cs_{s,i}^{} \widetilde N_n^\Delta (\dd s, \dd u),\\
				M_{t,i}^{\prime,n} &\defeq  -i \int_{[0,t) \times [\sigma_n^{},\infty)} u \cs_{s,i}^{}  \widetilde \Ns'_n (\dd s, \dd u).
			\end{align*}
			Because we have defined $N_n$ as the image of $N_n^\ast$ under the map $(t,u) \mapsto (t,u/\sigma_n^{})$, we have
			$\sigma_n^{} M_{t,i}^n = M_{t,i}^{\ast,n}$. Consequently, 
			\begin{equation*}
				\sup_{t \in [0,T]} | \sigma_n^{} M_{t,i}^n -  M_{t,i}^{} | 
				\leqslant
				\sup_{t \in [0,T]} |M_{t,i}^{\Delta,n}| + \sup_{t \in [0,T]}|M_{t,i}^{\prime,n}|. 
			\end{equation*}
			In order to not overload the notation we are going to drop the superscript $n$ in the remainder of the proof.
			
			To bound the first supremum, we further split $M_{t,i}^{\Delta} = M_{t,i}^{\Delta,<} + M_{t,i}^{\Delta,\geqslant}$ with
			\begin{equation*}
				M_{t,i}^{\Delta,<} \defeq  -i \int_{[0,t) \times (0,\sigma_n^{1/2})} u \cs_{s,i}^{} \widetilde N_n^\Delta (\dd s, \dd u)
			\end{equation*} 
			and
			\begin{equation*}
				M_{t,i}^{\Delta,\geqslant} \defeq  -i \int_{[0,t) \times [\sigma_n^{1/2} ,\sigma_n^{})} u \cs_{s,i}^{} \widetilde N_n^\Delta (\dd s, \dd u).
			\end{equation*} 
			By Doob's inequality and Ito isometry, we have
			\begin{equation*}
				\PP{\sup_{t \in [0,T]} |M_{t,i}^{\Delta,<}| \geqslant \varepsilon / 3 } \leqslant
				9 \varepsilon^{-2} \EE{\big (M_{T,i}^{\Delta,<}\big )^2}
				\leqslant 
				9 \varepsilon^{-2} i^2 T
				\int_0^{\sigma_n^{1/2}} u^{\alpha - 1}  \Big |1 - \Big ( 1 - \frac{u}{\sigma_n^{}}   \Big )^{\beta - 1}  \Big | \dd u.
			\end{equation*}
			As in the proof of Lemma~\ref{lem:integralestimate2}, we use the mean value theorem to bound
			\begin{equation*}
				\Big | 1 - \big ( 1 - x  \big )^{\beta - 1} \Big | \leqslant K x,
			\end{equation*}
			for some constant $K$ and all $x \in [0,1/2]$, say. Then, for $n$ sufficiently large so that $\sigma_n^{1/2} \leqslant \sigma_n^{} / 2$, we can bound the right-hand side as follows.
			\begin{equation*}
				9 \varepsilon^{-2} i^2 T
				\int_0^{\sigma_n^{1/2}} u^{\alpha - 1}  \Big |1 - \Big ( 1 - \frac{u}{\sigma_n^{}}   \Big )^{\beta - 1}  \Big | \dd u
				\leqslant
				9 \varepsilon^{-2} i^2 T K \sigma_n^{-1} \int_0^{\sigma_n^{1/2}} u^\alpha \dd u = O\big (\sigma_n^{(\alpha - 1)/2} \big ),
			\end{equation*}
			which shows that
			\begin{equation} \label{mdelta<estimate}
				\PP{\sup_{t \in [0,T]} |M_{t,i}^{\Delta,<}| \geqslant \varepsilon / 3 } = O \big (\sigma_n^{(\alpha -  1)/2} \big ).
			\end{equation}
			Next, we deal with $M_{t,i}^{\Delta, \geqslant}$. We decompose
			\begin{equation*}
				M_{t,i}^{\Delta, \geqslant} 
				=
				-i \int_{[0,t) \times [\sigma_n^{1/2} ,\sigma_n^{})} u \cs_{s,i}^{}  N_n^\Delta (\dd s, \dd u) 
				+
				i \int_{[0,t) \times [\sigma_n^{1/2} ,\sigma_n^{})} u \cs_{s,i}^{} \mu_n^\Delta (\dd s, \dd u).
			\end{equation*}
			On the complement of the event
			\begin{equation*}
				A_{n, T} \defeq \big  \{ (\Ns_n' \cup N_n^\Delta) \cap [0,T) \times [\sigma_n^{1/2},\infty) \neq \varnothing  \big \},
			\end{equation*}
			the first integral vanishes, and after substituting $v = \sigma_n^{-1} u$, the second one is bounded by
			\begin{equation*}
				i T \sigma_n^{\alpha - 1}  \int_{0}^{1/2} v^{\alpha - 2}  \big | 1 -  ( 1 - v)^{\beta - 1}  \big |  \dd s \dd u + i T \sigma_n^{\alpha - 1}  \int_{1/2}^{1} v^{\alpha - 2}  \big | 1 -  ( 1 - v)^{\beta - 1}  \big |  \dd s \dd u .
			\end{equation*}
			Here, the first integral is of order $O(\sigma_n^{\alpha -1})$ because
			$|1 - (1-v)^{\beta - 1}| \leqslant K v$ on $[0,1/2)$ and the second one 
			is of order $O(\sigma_n^{\alpha -1})$ since
			\begin{equation} \label{largevestimate}
				\big | 1 -  ( 1 - v)^{\beta - 1}  \big | 
				\leqslant
				2 + 2 (1-v)^{\beta - 1} \quad \textnormal{for all } v \in (1/2,1).
			\end{equation}
			We have thus shown that
			\begin{equation*}
				\limsup_{n \to \infty} \PP{\sup_{t \in [0,T]} |M_{t,i}^{\Delta, \geqslant} | \geqslant \varepsilon / 3}   
				\leqslant
				\limsup_{n \to \infty} \PP{A_{n,T}}.
			\end{equation*}
			A straightforward calculation shows that $\PP{A_{n,T}}\to 0$ as $n \to \infty$. 
			Indeed, for any Poisson random variable $Z$ with parameter $\lambda$, one has $\PP{X \geqslant 1} = 1 - \ee^{-\lambda} \leqslant \lambda$. Hence,
			\begin{equation*}
				\begin{split}
					&\PP{A_{n,T}} \leqslant \PP{\Ns_n' \cap \big ( [0,T) \times [\sigma_n^{1/2},\infty) \big ) \neq \varnothing}
					+
					\PP{N_n^\Delta \cap \big ( [0,T) \times [\sigma_n^{1/2},\infty) \big ) \neq \varnothing} \\
					& \quad \leqslant
					T \int_{\sigma_n^{1/2}}^\infty u^{\alpha - 3} \dd u 
					+
					T \int_{\sigma_n^{1/2}}^\infty 
					\left | \Big ( 1 - \frac{u}{\sigma_n^{}}    \Big )^{\beta - 1} - 1  \right | u^{\alpha-3}  \1_{(0,\sigma_n^{})} (u)  \dd u.
				\end{split}
			\end{equation*}
			The first integral obviously vanishes as $n \to \infty$. After substituting $v = u \sigma_n^{-1}$, the second integral reads
			\begin{equation*}
				\sigma_n^{\alpha - 2} \int_{\sigma_n^{-1/2}}^1 
				| (1-v)^{\beta - 1} - 1 | v^{\alpha - 3} \dd v.
			\end{equation*}
			As previously, we use for $v \leqslant 1/2$ that 
			$|(1-v)^{\beta - 1} - 1| \leqslant K v$ for some $K$, while for 
			$v \in (1/2,1]$ we use the estimate in Eq.~\eqref{largevestimate}. 
			This shows that for some (different) constant $K$, 
			\begin{equation*}
				\sigma_n^{\alpha - 2} \int_{\sigma_n^{-1/2}}^1 
				| (1-v)^{\beta - 1} - 1 | v^{\alpha - 3} \dd v
				\leqslant K \sigma_n^{\alpha - 2} \sigma_n^{(1-\alpha)/2} \to 0
				\quad \textnormal{as } n \to \infty,
			\end{equation*}
			where we absorbed some additive constants into $K$. 
			
			Analogously, one can show that
			\begin{equation*}
				\lim_{n \to \infty} \PP{\sup_{t \in [0,T]} |M_{t,i}' | \geqslant \varepsilon / 3}   
				=
				\lim_{n \to \infty} \PP{A_{n,T}} 
				=
				0.
			\end{equation*}
			Together with Eq.~\eqref{mdelta<estimate}, we finally obtain
			\begin{equation*}
				\begin{split}
					&\lim_{n \to \infty} \PP{\sup_{t \in [0,T]} |\sigma_n^{} M_{t,i}^n -i  M_{t,i}^{}| \geqslant \varepsilon } \\
					& \quad \leqslant 
					\lim_{n \to \infty} 
					\Bigg (
					\PP{\sup_{t \in [0,T]} |M_{t,i}'| \geqslant \varepsilon / 3 } 
					+
					\PP{\sup_{t \in [0,T]} |M_{t,i}^{\Delta,<}| \geqslant \varepsilon / 3 }
					+
					\PP{\sup_{t \in [0,T]} |M_{t,i}^{\Delta,\geqslant}| \geqslant \varepsilon / 3}
					\Bigg ) \\
					& \quad =
					0.
				\end{split}
			\end{equation*}
		\end{proof}
		
		We need another a-priori estimate for the size of the fluctuations.
		\begin{lemma} \label{lem:fluctuationsaresmallenough}
			For all $1 \leqslant i \leqslant d$ and $\varepsilon, T > 0$ we have
			\begin{equation*}
				\lim_{n \to \infty} \PP{\sup_{t \in [0,T]} \sigma_n^{} (C_{t,i}^n - \cs_{t,i}^{})^2 \geqslant \varepsilon} = 0.
			\end{equation*}
		\end{lemma}
		\begin{proof}
			Following the proof of Theorem~\ref{thm:lln} after Eq.~\eqref{Ddef}
			\begin{equation*}
				\sup_{t \in [0,T]} \sigma^{1/2}_n D_t^n  \leqslant K \int_0^T \sigma_n^{1/2} D_s^n \dd s 
				+
				\max_{1 \leqslant i \leqslant d} \sigma_n^{1/2} |\widehat M_{T,i}^n| + O(n^{\alpha - 1 + \frac{1-\alpha}{4 - 2\alpha}})
			\end{equation*}
			and thus, for an appropriate $\varepsilon' > 0$ by Grönwall's inequality
			\begin{equation*}
				\PP{\sup_{t \in [0,T]}\sigma_n^{1/2} D_t^n \geqslant \varepsilon} 
				\leqslant
				(\varepsilon')^{-2} \max_{1 \leqslant i \leqslant d} \sigma_n^{} \EE{(\widehat M_{T,i}^n)^2} +  O(n^{\alpha - 1 + \frac{1-\alpha}{4 - 2\alpha}}) 
				=   O(n^{\alpha - 1 + \frac{1-\alpha}{2 - \alpha}}),
			\end{equation*}
			which holds for $n$ large enough and the result follows by taking $n \to \infty$.
		\end{proof}
		Now, we just have to put the pieces together.
		\begin{proof}[Proof of Theorem~\ref{thm:fclt}]
			Recall the decomposition from \eqref{decomp}
			\begin{equation*}
				C_{t,i}^n = \delta_{i,1} +  \int_0^t F_i^n (\bC_s^n) \dd s +  \widehat M_{t,i}^n 
			\end{equation*}
			and that $U_{t,i}^n = \sigma_n^{} \big ( C_{t,i}^n   - \cs_{t,i}^{}    \big )$ by definition. Recall that $\cs_i^{}$ satisfies the ordinary differential equation $\frac{\dd}{\dd t} \cs_i^{} = F_i (\cs)$  (see Eq.~\eqref{eq:system_ODE}) with initial conditional $\cs_{0,i}^{} = \delta_{i,1}$. Hence, $\cs_i^{}$ satisfies (for all $t \geqslant 0$) the integral equation
			\begin{equation*}
				\cs_{t,i}^{} = \delta_{i,1} + \int_0^t F_i (\cs_s^{}) \dd s.
			\end{equation*}
			Together with the above, this implies that
			\begin{equation*}
				U_{t,i}^n = \sigma_n^{} \int_0^t F_i^n (\bC_s^n) - F_i^{} (\cs_s^{}) \dd s + \sigma_n^{} \widehat M_{t,i}^n
			\end{equation*}
			and therefore, recalling the definition of $U_{t,i}$ given in 
			Theorem~\ref{thm:fclt},
			\
			\begin{equation*}
				U_{t,i} = \int_0^t \langle \nabla F_i^{} (\cs_s^{}), \bU_s^{} \rangle \dd s +  M_{t,i},
			\end{equation*}
			we see that
			\begin{equation} \label{eq:ß38j94f}
				\begin{split}
					U_{t,i}^n - U_{t,i}^{}&=
					\int_0^t \Big ( \sigma_n^{} \big ( F_i^{} (\bC_s^n) - F_i^{} (\cs_s^{}) \big ) - \langle \nabla F_i^{} (\cs_s^{}), \bU_s^{} \rangle  \Big ) \dd s 
					+ \sigma_n^{} \int_0^t F_i^n (\bC_s^n) - F_i^{} (\bC_s^n) \dd s  \\
					&\quad + \sigma_n^{} \widehat M_{t,i}^n - M_{t,i}^{}\\
					&  = 
					\int_0^t \Big ( \sigma_n^{} \big ( F_i^{} (\bC_s^n) - F_i^{} (\cs_s^{}) \big ) - \langle \nabla F_i^{} (\cs_s^{}), \bU_s^{} \rangle \Big ) \dd s  
					\\&\quad + \sigma_n^{} \widehat M_{t,i}^n -
					M_{t,i}^{} 
					+ 
					O \big (n^{\alpha - 1 + \frac{1-\alpha}{2-\alpha}} \big),
				\end{split}
			\end{equation}
			where we used in the last step that, by Lemma~\ref{lem:uniformapproximation} and recalling that
			$\sigma_n^{} = n^{\frac{1-\alpha}{2-\alpha}}$,
			\begin{equation*}
				\Big | \sigma_n^{} \int_0^t F_i^n (\bC_s^n) - F_i^{} (\bC_s^n) \dd s \Big |
				\leqslant
				tn^{\frac{1-\alpha}{2-\alpha}} \sup_{\bc \in E^n} \big | F_i^n (\bc) - F_i^{}(\bc)   \big | = tn^{\frac{1-\alpha}{2-\alpha}} O (n^{\alpha - 1}).
			\end{equation*}
			Because $F_i^{}$ is a polynomial (and thus smooth), it is an immediate consequence of Taylor's theorem that
			\begin{equation*}
				\sigma_n^{} \big ( F_i^{} (\bC_s^n) - F_i^{} (\cs_s^{}) \big ) = \langle \nabla F_i^{} (\cs_s^{}), \bU_s^n \rangle + \mathbf{G}_s^{}, 
			\end{equation*}
			where
			\begin{equation*}
				|\mathbf{G}_t^{}| \leqslant K \sigma_n^{} \|\bC_t^n - \cs_t \|_\infty^2, 
			\end{equation*}
			for some uniform constant $K$, where $\| . \|_\infty$ denotes the maximum over all $d$ components. 
			Combining this with Eq.~\eqref{eq:ß38j94f} yields
			\begin{equation*}
				\begin{split}
					\big | U_{t,i}^n - U_{t,i} \big | 
					&\leqslant
					\Big | \int_0^t \langle \nabla F_i (\cs_s^{}), \bU^n_s - \bU_s^{} \rangle 
					\dd s \Big |
					+
					\int_0^t \big | \mathbf{G}_s \big | \dd s +
					\big | \sigma_n^{} \hat M_{t,i}^n - M_{t,i}^{}   \big | 
					+
					O \big ( n^{\alpha - 1 + \frac{1-\alpha}{2-\alpha}} \big ) \\
					& \leqslant
					\int_0^t \big | \langle
					\nabla F_i(\cs_s^{}), \bU_s^n - \bU_s^{} \rangle  \big | \dd s 
					+
					K \int_0^t \sigma_n^{} \| C_s^n - \cs_s^{} \|_\infty^2 \dd s   
					+
					\big | \sigma_n^{} \hat M_{t,i}^n - M_{t,i}^{}   \big | 
					+
					O \big ( n^{\alpha - 1 + \frac{1-\alpha}{2-\alpha}} \big ) \\
					& \leqslant
					K \int_0^t \big |  \bU_s^n - \bU_s^{} \big | \dd s 
					+
					K \int_0^t \sigma_n^{} \| C_s^n - \cs_s^{} \|_\infty^2 \dd s    
					+
					\| \sigma_n^{} \hat \bM_{t}^n - \bM_{t}^{}   \|_\infty 
					+
					O \big ( n^{\alpha - 1 + \frac{1-\alpha}{2-\alpha}} \big ),
				\end{split}
			\end{equation*}
			
			\color{black}
			where we used in the third step that 
			$\nabla F_i$ is bounded. Note that the right-hand side does not depend on $i$. Taking on the left the maximum over all $i$ and
			setting
			$\Us_t^n \defeq \| \bU_t^n - \bU_t^{} \|_\infty$, we see that 
			\begin{equation*}
				\Us_t^n \leqslant K \int_0^t \Us_s^n \dd s 
				+ \sup_{t,r \in [0,T]} K \Big (\|  \sigma_n^{} \widehat \bM_{t}^n - \bM_{t}^{}\|_\infty +  \sigma_n^{} \|\bC_r^n - \cs_r \|_\infty^2 +  n^{-\varrho}    \Big ) 
				+ O(n^{\alpha - 1 + \frac{1 - \alpha}{2 - \alpha}}) 
			\end{equation*}
			with $\varrho \defeq \alpha - 1 + \frac{1-\alpha}{2 - \alpha}$.
			Clearly, $\Us_t^n \leqslant 2 \sigma_n^{} + $
			By Grönwall's inequality, we see that for sufficiently large $n$
			\begin{equation*}
				\PP{\sup_{t \in [0,T]} \Us_t^n \geqslant \varepsilon} \leqslant 
				\sum_{i=1}^d \PP{ \sup_{t,r \in [0,T]} K \Big (| \sigma_n^{} \widehat M_{t,i}^n - M_{t,i}^{}| + \sigma_n^{} \|\bC_r^n - \cs_r \|_\infty^2     \Big )  \geqslant \frac{\varepsilon}{2 + 2\ee^{KT}}}.
			\end{equation*}
			By Lemmas~\ref{lem:replacetheintegrand},~\ref{lem:replacingthenoise} and~\ref{lem:fluctuationsaresmallenough}, the right-hand side goes to $0$ as $n \to \infty$.
		\end{proof}
		
		\section{Some auxiliary results} \label{sect:calculus}
		
		In the next two lemmas, we provide approximation results for certain integrals that appear throughout the manuscript.
		
		\begin{lemma} \label{lem:integralestimate1}
			For all $\theta \in \Rb$, $\alpha \in (0,1)$ and all $k \geqslant 2$, 
			\begin{equation*}
				\int_0^1 u^{k + \alpha - 3} (1 - u)^{n + \theta} \dd u = n^{2 - \alpha - k} \Gamma(k+ \alpha - 2) + O(n^{1-\alpha-k}).
			\end{equation*}
			as $n \to \infty$.
		\end{lemma}
		\begin{proof}
			Since $k\geqslant 2$, we have $k+\alpha-3\geqslant -1$, hence by the definition of the Beta-function
			\begin{align*}
				\int_0^1 u^{k + \alpha - 3} (1 - u)^{n + \theta} \dd u = \frac{\Gamma(k+\alpha-2) \Gamma(n+\theta+1)}{\Gamma(n+\theta+k+\alpha-2)}.
			\end{align*}
			We note that by 6.1.47 in \cite{Abramowitz1972} it holds
			\begin{align*}
				\frac{\Gamma(n+\theta+1)}{\Gamma(n+\theta+k+\alpha-2)} (n+\theta)^{k+\alpha-2} = 1+ O(n^{-1}),
			\end{align*}
			and we arrive at
			\begin{align*}
				\int_0^1 u^{k + \alpha - 3} (1 - u)^{n + \theta} \dd u &= \Gamma(k+\alpha-2) (n+\theta)^{2-\alpha-k} (1+O(n^{-1})) \\
				&=\Gamma(k+\alpha-2) n^{2-\alpha-k} (1+O(n^{-1})).
			\end{align*}
		\end{proof}
		
		\begin{lemma} \label{lem:integralestimate2}
			For all $\theta_1^{} \in (-1, \infty), \theta_2^{} \in \Rb$ and $\alpha \in (0,1)$,
			\begin{equation*}
				\int_0^1 u^{\alpha - 2} (1 - u)^{\theta_1^{}}  \big ( 1 - (1 - u)^{n + \theta_2^{}}        \big ) \dd u 
				=
				n^{1 - \alpha} \frac{\Gamma(\alpha)}{1 - \alpha} + O(1),
			\end{equation*}
			as $n \to \infty$. 
		\end{lemma}
		\begin{proof}
			We start by decomposing the integral as
			\begin{align}
				\int_0^{1/2}       u^{\alpha - 2} (1 - u)^{\theta_1^{}}  \big ( 1 - (1 - u)^{n + \theta_2^{}}        \big ) \dd u 
				+
				\int_{1/2}^1 u^{\alpha - 2} (1 - u)^{\theta_1^{}}  \big ( 1 - (1 - u)^{n + \theta_2^{}}        \big ) \dd u .  \label{eq:decomposition_int}
			\end{align}
			Clearly, since $ \big ( 1 - (1 - u)^{n + \theta_2^{}}        \big )  \leqslant 1$ for $n \geqslant -\theta_2^{}$, the second integral in \eqref{eq:decomposition_int} can be bounded uniformly in $n$. To deal with the first integral in \eqref{eq:decomposition_int},
			we employ the substitution $u \mapsto un$ and obtain
			\begin{equation*}
				\int_0^{1/2}       u^{\alpha - 2} (1 - u)^{\theta_1^{}}  \big ( 1 - (1 - u)^{n + \theta_2^{}}        \big ) \dd u 
				=
				n^{1 - \alpha} 
				\int_0^{n/2} u^{\alpha - 2} \Big ( 1 - \frac{u}{n}  \Big )^{\theta_1^{}} \Big ( 1 - \Big (1 - \frac{u}{n}          \Big )^{n + \theta_2^{}}           \Big ) \dd u.
			\end{equation*}
			We split the integral on the right-hand side into three parts
			\begin{equation} \label{decomposition2}
				\begin{split}
					&\int_0^{n/2} u^{\alpha - 2} \Big ( 1 - \frac{u}{n}  \Big )^{\theta_1^{}} \Big ( 1 - \Big (1 - \frac{u}{n}          \Big )^{n + \theta_2^{}}           \Big ) \dd u \\
					& \quad = 
					\int_0^{n/2} u^{\alpha - 2} \Big ( 1 - \frac{u}{n}  \Big )^{\theta_1^{}} ( 1 - \ee^{-u}) \dd u  +
					\int_0^{n/2} u^{\alpha - 2} \Big ( 1 - \frac{u}{n}  \Big  )^{\theta_1^{}} \Big ( \ee^{-u} - \Big ( 1 - \frac{u}{n}  \Big )^n             \Big ) \dd u \\ 
					& \quad \quad +
					\int_0^{n/2} u^{\alpha - 2} \Big ( 1 - \frac{u}{n}  \Big  )^{\theta_1^{} + n} \Big ( 1 - \Big (1 - \frac{u}{n}   \Big )^{\theta_2^{}}  \Big ) \dd u.
				\end{split}
			\end{equation}
			To proceed, note that the mean-value theorem implies that 
			\begin{equation} \label{meanvalue}
				\left | 1 - \Big (1 - \frac{u}{n}   \Big )^{\theta_2^{}}     \right | \leqslant K \frac{u}{n}
			\end{equation}
			for all $u \in [0, n/2]$ and some constant $K > 0$; here and in the following, $K$ will always denote a constant that may change its value from line to line. Thus, the third integral in Eq.~\eqref{decomposition2} can be bounded from above by
			\begin{equation*}
				K n^{-1} \int_0^{n/2} u^{\alpha - 1} \Big ( 1- \frac{u}{n}   \Big )^n \dd u \leqslant K n^{-1} \int_0^\infty u^{\alpha -1} \ee^{-u} \dd u = O(n^{-1}),
			\end{equation*}
			again for some constant $K$ (different from the one above).
			To bound the second integral, we use Lemma~\ref{lem:approximatingtheexponential} and see that
			\begin{equation*}
				\int_0^{n/2} u^{\alpha - 2} \Big ( 1 - \frac{u}{n}  \Big  )^{\theta_1^{}} \Big ( \ee^{-u} - \Big ( 1 - \frac{u}{n}  \Big )^n             \Big ) \dd u
				\leqslant
				Kn^{-1} \int_0^{\infty} u^\alpha \ee^{-u} \dd u = O(n^{-1}).
			\end{equation*}
			We further decompose the first integral in Eq.~\eqref{decomposition2} as follows
			\begin{equation*}
				\begin{split}
					& \int_0^{n/2} u^{\alpha - 2} \Big ( 1 - \frac{u}{n}  \Big )^{\theta_1^{}} ( 1 - \ee^{-u}) \dd u \\
					& \quad = \int_0^{\infty} u^{\alpha - 2} (1- \ee^{-u}) \dd u - \int_{n/2}^\infty u^{\alpha - 2} (1- \ee^{-u}) \dd u \\
					& \quad \quad + \int_0^{n/2} u^{\alpha - 2} \Big ( \Big ( 1 - \frac{u}{n} \Big)^{\theta_1^{}}  - 1   \Big  ) ( 1 - \ee^{-u}) \dd u.
				\end{split}
			\end{equation*}
			The first integral is precisely $\Gamma(\alpha) / (1-\alpha)$, as can be seen by an elementary application of integration by parts. Bounding 
			$1 - \ee^{-u}$ in the second integral by $1$, we see that it is bounded in absolute value by $n^{\alpha - 1}$. For the third integral, we 
			once more apply Eq.~\eqref{meanvalue} to see that
			\begin{equation*}
				\int_0^{n/2} \left | u^{\alpha - 2} \Big ( \Big ( 1 - \frac{u}{n} \Big)^{\theta_1^{}}  - 1   \Big  ) ( 1 - \ee^{-u}) \right | \dd u
				\leqslant
				Kn^{-1} \int_0^{n/2} u^{\alpha - 1} \dd u
				= O(n^{\alpha  - 1}). 
			\end{equation*}
			
		\end{proof}
		\begin{lemma}\label{lem:approximatingtheexponential}
			Let $n \in \mathbb{N}$, then for all $u \in [0, \frac{n}{2}]$ it holds
			\begin{align*}
				\left  | \Big ( 1 - \frac{u}{n}   \Big )^n   - \ee^{-u}           \right | \leqslant 2 n^{-1} u^2 e^{-u}.
			\end{align*}
		\end{lemma}
		\begin{proof}
			Note that $\Big( 1 - \frac{u}{n} \Big)^n \leqslant e^{-u} $, hence we have
			\begin{align*}
				0 \leqslant f(u) := e^{-u} - \Big( 1 - \frac{u}{n} \Big)^n.
			\end{align*}
			Then, for all $u \in [0,\frac{n}{2}]$
			\begin{align*}
				f'(u)&= \Big( 1 - \frac{u}{n} \Big)^{n-1}- e^{-u} = \frac{\left( 1 - \frac{u}{n} \right)^n - e^{-u} \left( 1- \frac{u}{n} \right) }{ \left( 1 - \frac{u}{n} \right)} \\
				&\leqslant 2 \left[  e^{-u} - e^{-u} \left( 1- \frac{u}{n} \right) \right] = 2 e^{-u} \frac{u}{n}.
			\end{align*}
			Finally, noting that $f(0)=0$, by the mean value theorem
			\begin{align*}
				f(u) \leqslant u \sup_{0 \leqslant v \leqslant u} f'(v) \leqslant 2 e^{-u} u^2 n^{-1},
			\end{align*}
			which finishes the proof.
		\end{proof}
		
		\subsection*{Acknowledgements}
		Frederic Alberti was funded by the Deutsche Forschungsgemeinschaft (DFG, German Research Foundation) --- Project-ID 519713930.
		
		\noindent Fernando Cordero was funded by the Deutsche Forschungsgemeinschaft (DFG, German Research Foundation) --- Project-ID 317210226 --- SFB 1283.

		\appendix
		\section{Derivation of identity \eqref{samplingformula1}} \label{app-bubble}
		In this section we explain the derivation of identity \eqref{samplingformula1}. 
		Recall that $S_0$ is a uniformly chosen subset of $[n]$ with $|S_0| = m_0^{}$.
		Fix $\bl \in \Nb^d$ with $\bl \leqslant \bc$.
		To determine 
		$\PP{\tla_t^n = \bl \, | \, \tC_t = \bc}$, we will count the number of realisations of $S_0$, which intersect, for each $i \in [n]$, precisely $\ell_i^{}$ components of size $i$. Put differently, we are asking how many ways there are to sprinkle $m_0^{}$ marks onto the underlying graph, so that, for each $i \in [n]$, exactly $\ell_i^{}$ components of size $i$ contain some marked vertex.
		
		The first step is to choose, for each 
		$i \in [n]$, $\ell_i^{}$ out of the $c_i^{}$ components that will receive at least one mark. Clearly, this can be done in
		\begin{equation}
			\label{binomialfactor}
			\prod_{i \in [n]} \binom{c_i^{}}{\ell_i^{}}
		\end{equation}
		different ways. Once we have fixed the components that should receive marks, we count the number of `sprinklings' that are compatible with this choice. For $i \in [n]$, denote the $\ell_i^{}$ chosen blocks of size $i$ by
		$A_{i,1},A_{i,2},\ldots,A_{i,\ell_i^{}}$.
		Given a compatible sprinkling $S_0$, we set for each $i \in [n]$
		and $j \in [\ell_i^{}]$ 
		\begin{equation*}
			f_{i,j}^{} \defeq 
			\# \big  \{k \in A_{i,j} : k \leqslant 
			\min(A_{i,j} \cap S_0) \big \}.
		\end{equation*}
		Note that $f_{i,j}^{} \in [i]$ for all 
		$i \in [n]$ and $j \in [\ell_i^{}]$. 
		
		For now, we regard the values of $f_{i,j}^{}$ as fixed and count all compatible realisations of $S_0$. More precisely, this means that we are counting the number of realisations of $S_0$ satisfying the following two conditions.
		\begin{enumerate}[label=(\roman*)]
			\item 
			Writing 
			$A_{i,j} \defeq 
			\{ v_1^{},\ldots,v_i^{}   \}$ 
			with
			$v_1^{} < \ldots < v_k^{}$, and letting
			$f \defeq f_{i,j}^{}$,
			we demand that $v_f^{}$ carry a mark, meaning that
			$v_f^{} \in S_0$.
			\item 
			On the other hand, smaller vertices do not carry any mark, hence
			$\{ v_1^{},\ldots,v_{f-1}^{} \} \cap S_0
			=
			\varnothing.
			$
		\end{enumerate}
		In how many ways can we distribute marks under the restrictions imposed by (i) and (ii)?
		In accordance with (i), we put a mark on the $f_{i,j}$-st vertex (in ascending order) in $A_{i,j}$, for each $i \in [n]$ and 
		$j \in [\ell_i^{}]$. This leaves 
		$m_0^{} - \ell_1^{} - \ldots - \ell_n^{}$
		marks yet to be distributed. But first note that condition (ii) excludes 
		\begin{equation*}
			\sum_{i=1}^n \sum_{j =1}^{\ell_i^{}} 
			\big ( f_{i,j} - 1 \big )
		\end{equation*}
		vertices from carrying any mark. Also, 
		$|\bl| = \ell_1^{} + \ldots + \ell_n^{}$ vertices already carry a mark due to (i). This leaves us with 
		$m_0^{} - |\bl|$ marks that can be freely sprinkled over 
		$| \bl | -  f_{\bl}^{}$ vertices with $f_{\bl}^{} \defeq 
		\sum_{i=1}^n \sum_{j=1}^{\ell_i^{}} f_{i,j}^{}$. 
		Obviously, this can be done in 
		\begin{equation*}
			\binom{ {|\bl| - f_{\bl}^{} } }{ {m_0^{} - |\bl|} }
		\end{equation*}
		different ways, giving the total number of  realisations of $S_0$ compatible with 
		$\tla_t^n = \bl$, any given choice of the $A_{i,j}$ and fixed values of
		$f_{i,j}^{}$.
		
		To finish our counting, it remains to multiply this with the binomial factor from \eqref{binomialfactor} reflecting the choice of the $A_{i,j}$ and to sum over all possible values for $f_{i,j}^{}$. After dividing by the total number of possible realisations, we obtain
		\begin{equation}
			\label{summingovereffs}
			\PP{\tla_t^{n}=\bl\mid \tC_t^n=\bc} 
			=
			\frac{
				\prod_{i \in [n]} \binom{c_i^{}}{\ell_i^{}}
				\sum \binom{|\bl| - f_{\bl}^{}}{m_0^{} - |\bl|}
			}{\binom{n}{m_0^{}}}.
		\end{equation}
		Recalling that $f_{i,j}^{}$ refers to an arbitrary vertex in $A_{i,j}^{}$, one of the $\ell_i^{}$ components chosen in the first step, it is clear that
		the sums in the numerator runs over all choices of
		$f_{i,j} \in [i]$ for $i \in [n]$ and 
		$j \in [\ell_i^{}] $.
		
		To write this more succinctly, let $(F_{i,j})_{i \in [n], j \in [\ell_i^{}]}$ be a family of independent random variables, where for fixed $i$, $F_{i,j}^{}$ is uniformly distributed on $[i]$ and we set $F_{\bl} \defeq \sum_{i=1}^n\sum_{j=1}^{\ell_i^{}} F_{i,j}$. With this, we can rewrite the sum in \eqref{summingovereffs} as
		\begin{equation}
			\label{sumasexpectation}
			\sum 
			\binom{\| \bl\|-f_{\bl}}{m_0^{} - |\bl|}
			=
			\prod_{i=1}^n i^{\ell_i^{}} 
			\EE{
				\binom{\|\bl\|-F_{\bl}}{m_0^{} - |\bl|}
			}.
		\end{equation}
		The factor $\prod_{i=1}^n i^{\ell_i^{}}$ reflects the fact that the uniform distribution puts a point mass $\prod_{i=1}^n i^{-\ell_i^{}}$ on each possible realisation of the $F_{i,j}$.
		Inserting \eqref{sumasexpectation} back into
		\eqref{summingovereffs} finally yields the desired formula \eqref{samplingformula1}.

		\bibliographystyle{alpha}
		\bibliography{bibfile}
		
	\end{document}